\documentclass{article}
\usepackage{amsmath,amssymb,amsthm}
\usepackage{cite}

\begin{document}

\begin{center}
{\bf Explicit determinantal  formulas for  solutions to the generalized Sylvester quaternion matrix equation  and   its special cases.}\end{center}
\begin{center}{ \textbf{\emph{Ivan I. Kyrchei} }\footnote{kyrchei@online.ua,\\ Pidstrygach Institute for Applied Problems of Mechanics and Mathematics of NAS of Ukraine,
 Ukraine} }\end{center}

\begin{abstract}
Within the framework of the theory of quaternion column-row determinants and
 using  determinantal representations of  the Moore-Penrose inverse previously obtained by the author, we get  explicit determinantal representation formulas  of  solutions
 (analogs of Cramer's rule) to the quaternion  two-sided generalized Sylvester matrix equation $  {\bf A}_{1}{\bf X}_{1}{\bf B}_{1}+ {\bf A}_{2}{\bf X}_{2}{\bf B}_{2}={\bf C}$ and its all special cases when its first term or both terms are one-sided. Finally, we derive determinantal representations of two like-Lyapunov equations.  \end{abstract}


\newtheorem{thm}{Theorem}[section]
 \newtheorem{cor}[thm]{Corollary}
 \newtheorem{lem}[thm]{Lemma}
 \newtheorem{prop}[thm]{Proposition}
 \theoremstyle{definition}
 \newtheorem{defn}[thm]{Definition}
 \theoremstyle{remark}
 \newtheorem{rem}[thm]{Remark}
 \newtheorem*{ex}{Example}
 \numberwithin{equation}{section}
\newcommand{\rank}{\mathop{\rm rank}\nolimits}

\section{Introduction}

Let  $ {\mathbb{H}}^{m\times n}$ and $ {\mathbb{H}}^{m\times n}_{r}$  stand for  the set of all  $m\times n$ matrices  and matrices  with  rank $r$, respectively, over the quaternion skew field
$ {\mathbb{H}}=\{a_{0}+a_{1}i+a_{2}j+a_{3}k\,
|\,i^{2}=j^{2}=k^{2}=-1,\, a_{0}, a_{1}, a_{2}, a_{3}\in{\rm
{\mathbb{R}}}\},$
where $ {\mathbb{R}}$ is  the real number field. In this paper,  we  investigate  the two-sided coupled generalized Sylvester matrix equation  over  $ {\mathbb{H}}$,
\begin{equation}\label{eq:tss1}
 {\bf A}_{1}{\bf X}_{1}{\bf B}_{1}+ {\bf A}_{2}{\bf X}_{2}{\bf B}_{2}={\bf C}.
\end{equation}
Since Sylvester-type  matrix  equations  have  wide  applications in  several  fields (see, e.g.\cite{chen,shah,syrm1,syrm2,varg,zhang1}), these equations are thoroughly studied and there are many important results about them (see,e.g.\cite{deh,liao,peng,shi,wua,xu,zhou}).  Mansour \cite{mans} studied the solvability condition of  (\ref{eq:tss1}) in the operator algebra.
Liping \cite{lip} has began  investigations of a similar equation over  the quaternion skew field. Baksalary and Kala \cite{bak} derived  the general solution to (\ref{eq:tss1}) expressed in terms of generalized inverses that has been extended to the quaternion skew field in
\cite{wang04,wang09}. Quaternion matrix equations similar to  Eq.(\ref{eq:tss1}) have been recently investigated by many authors
(see, e.g. \cite{csong1,csong2,he2,he3,fut,reh1,reh2,reh3,sim,yuan1,yuan2,yuan3,zhan}).

The main goal of this paper is to derive  determinantal representations of the general
solution to Eq.(\ref{eq:tss1}) and its all special cases over the quaternion skew field  using previously obtained determinantal representations of the Moore-Penrose inverse. Evidently, determinantal representation of a solution of a matrix equation (which is expressed in terms of generalized inverses) gives a direct method of its finding analogous to  classical Cramer's rule (when a solution is expressed by an usual inverse) that has important theoretical and practical significance \cite{kyr_nov}.

 For $ {\bf A}
 \in  {\mathbb{H}}^{m\times n}$, the symbol $ {\bf A}^{ *}$ stands for  conjugate transpose (Hermitian adjoint)
of $ {\bf A}$.
 A matrix $ {\bf A}  \in
{\mathbb{H}}^{n\times n}$ is Hermitian if $ {\bf
A}^{ *}  =  {\bf A}$.
The  definition of the Moore-Penrose inverse matrix has been  extended  to quaternion matrices as follows.
\begin{defn}
The Moore-Penrose inverse of ${\bf A}\in{\rm {\mathbb{H}}}^{m\times n}$, denoted by ${\bf A}^{\dagger}$, is the unique matrix ${\bf A}^{\dagger}\in{\rm {\mathbb{H}}}^{n\times m}$ satisfying the following four equations
$$
\text{1. } {\bf A}{\bf A}^{\dagger}{\bf A}={\bf A},~\text{2. } {\bf A}^{\dagger}{\bf A}{\bf A}^{\dagger}={\bf A}^{\dagger},~\text{3. } ({\bf A}{\bf A}^{\dagger})^{*}={\bf A}{\bf A}^{\dagger},~\text{4. } ({\bf A}^{\dagger}{\bf A})^{*}={\bf A}^{\dagger}{\bf A}.
$$
\end{defn}
The problem for determinantal representation of quaternion generalized  inverses as well as solutions   and generalized inverse solutions of  quaternion matrix equations  only now can be solved due to the theory of column-row determinants introduced in \cite{kyr2,kyr3}.
Within the framework of the theory of column-row determinants, determinantal representations of various generalized quaternion inverses and generalized inverse solutions to quaternion matrix
equations have been derived by the author (see, e.g.\cite{kyr4,kyr5,kyr6,kyr7,kyr8,kyr9,kyr10,kyr11,kyr12,kyr13}) and by other researchers (see, e.g.\cite{song1,song2,song3,song5}).

The paper is organized as follows.  In Subsection 2, we start with   introduction of  row-column determinants, and determinantal representations of the Moore-Penrose inverse and of solutions to the quaternion matrix equation ${\bf A}{\bf X} {\bf B} =  {\bf C}$ and its special cases previously obtained within the framework of the theory   of  row-column determinants.  The explicit determinantal representation solution  to Eq.(\ref{eq:tss1})  is derived in  Section 3.  In Subsection 4, we give Cramer's rules to special cases of (\ref{eq:tss1}) when only one term of the equation is two-sided. In Subsection 5 and 6, we get Cramer's rules to special cases of (\ref{eq:tss1}) when both terms of the equation are one-sided, and to some two similar Lyapunov equations, respectively.
A numerical example to illustrate the main results is considered in   Section 7. Finally, in Section 8, the conclusions are drawn.

\section{Preliminaries. Elements of the theory of  row-column determinants.}

 For  ${\rm {\bf A}} \in
{\mathbb{H}}^{n\times n}$, we define $n$ row determinants and $n$ column determinants.
Suppose $S_{n}$ is the symmetric group on the set $I_{n}=\{1,\ldots,n\}$.
\begin{defn}
 \emph{The $i$th row determinant} of ${\rm {\bf A}}=(a_{ij}) \in
{\mathbb{H}}^{n\times n}$ is defined  for all $i =1,\ldots,n $
by putting
 \begin{gather*}{\rm{rdet}}_{ i} {\rm {\bf A}} =
\sum\limits_{\sigma \in S_{n}} \left( { - 1} \right)^{n - r}({a_{i{\kern
1pt} i_{k_{1}}} } {a_{i_{k_{1}}   i_{k_{1} + 1}}} \ldots   {a_{i_{k_{1}
+ l_{1}}
 i}})  \ldots  ({a_{i_{k_{r}}  i_{k_{r} + 1}}}
\ldots  {a_{i_{k_{r} + l_{r}}  i_{k_{r}} }}),\\
\sigma = \left(
{i\,i_{k_{1}}  i_{k_{1} + 1} \ldots i_{k_{1} + l_{1}} } \right)\left(
{i_{k_{2}}  i_{k_{2} + 1} \ldots i_{k_{2} + l_{2}} } \right)\ldots \left(
{i_{k_{r}}  i_{k_{r} + 1} \ldots i_{k_{r} + l_{r}} } \right),\end{gather*}
where $i_{k_{2}} < i_{k_{3}}  < \cdots < i_{k_{r}}$ and $i_{k_{t}}  <
i_{k_{t} + s} $ for all $t = 2,\ldots,r $ and  $s =1,\ldots,l_{t} $.
\end{defn}
\begin{defn}
\emph{The $j$th column determinant}
 of ${\rm {\bf
A}}=(a_{ij}) \in
{\mathbb{H}}^{n\times n}$ is defined for
all $j =1,\ldots,n $ by putting
 \begin{gather*}{\rm{cdet}} _{{j}}  {\bf A} =
\sum\limits_{\tau \in S_{n}} ( - 1)^{n - r}(a_{j_{k_{r}}
j_{k_{r} + l_{r}} } \ldots a_{j_{k_{r} + 1} j_{k_{r}} })  \ldots  (a_{j
j_{k_{1} + l_{1}} }  \ldots  a_{ j_{k_{1} + 1} j_{k_{1}} }a_{j_{k_{1}}
j}),\\
\tau =
\left( {j_{k_{r} + l_{r}}  \ldots j_{k_{r} + 1} j_{k_{r}} } \right)\ldots
\left( {j_{k_{2} + l_{2}}  \ldots j_{k_{2} + 1} j_{k_{2}} } \right){\kern
1pt} \left( {j_{k_{1} + l_{1}}  \ldots j_{k_{1} + 1} j_{k_{1} } j}
\right), \end{gather*}
\noindent where $j_{k_{2}}  < j_{k_{3}}  < \cdots <
j_{k_{r}} $ and $j_{k_{t}}  < j_{k_{t} + s} $ for  $t = {2,\ldots,r} $
and $s = {1,\ldots,l_{t}}  $.
\end{defn}
Since  \cite{kyr2} for Hermitian ${\rm {\bf A}}$ we have
 $${\rm{rdet}} _{1}  {\bf A} = \cdots = {\rm{rdet}} _{n} {\bf
A} = {\rm{cdet}} _{1}  {\bf A} = \cdots = {\rm{cdet}} _{n}  {\bf
A} \in  {\mathbb{R}},$$  its
determinant  is defined by  putting
$\det {\bf A}: = {\rm{rdet}}_{{i}}\,
{\bf A} = {\rm{cdet}} _{{i}}\, {\bf A} $
 for all $i =1,\ldots,n$.

We shall use the following notation. Let $\alpha : = \left\{
{\alpha _{1},\ldots,\alpha _{k}} \right\} \subseteq {\left\{
{1,\ldots ,m} \right\}}$ and $\beta : = \left\{ {\beta _{1}
,\ldots ,\beta _{k}} \right\} \subseteq {\left\{ {1,\ldots ,n}
\right\}}$ be subsets with $1 \le k \le \min {\left\{
{m,n} \right\}}$.
Let $ {\bf A}_{\beta} ^{\alpha} $ be a submatrix of $ {\bf A}
 \in  {\mathbb{H}}^{m\times n}$  whose rows are indexed by $\alpha$ and whose
columns are indexed by $\beta$. Similarly, let $ {\bf A}_{\alpha} ^{\alpha} $ be a principal submatrix of $ {\bf A}$ whose rows and columns are
indexed by $\alpha$.
 If $ {\bf A}$ is Hermitian, then
$ |{\bf A}|_{\alpha} ^{\alpha} $ denotes the
corresponding principal minor of $\det  {\bf A}$.
 For $1 \leq k\leq n$, the collection of strictly
increasing sequences of $k$ integers chosen from $\left\{
{1,\ldots ,n} \right\}$ is denoted by $\textsl{L}_{ k,
n}: = {\left\{ {\alpha :\alpha = \left( {\alpha _{1} ,\ldots
,\alpha _{k}} \right),\,{\kern 1pt} 1 \le \alpha _{1} < \cdots
< \alpha _{k} \le n} \right\}}$.  For fixed $i \in \alpha $ and $j \in
\beta $, let $I_{r, m} {\left\{ {i} \right\}}: = {\left\{
{\alpha :\alpha \in L_{r, m} ,i \in \alpha}  \right\}}$,
$J_{r, n} {\left\{ {j} \right\}}: = {\left\{ {\beta :\beta
\in L_{r, n}, j \in \beta}  \right\}}$.

 Let $ {\bf a}_{\,\textbf{.}j} $ be the $j$th column and $ {\bf
a}_{i\,\textbf{.}} $ be the $i$th row of $ {\bf A}$. Suppose $ {\bf
A}_{\,\textbf{.}j} \left(  {\bf b} \right)$ denotes the matrix obtained from
$ {\bf A}$ by replacing its $j$th column with the column ${\bf
b}$, and ${\bf A}_{i\,\textbf{.}} \left(  {\bf b} \right)$ denotes the
matrix obtained from $ {\bf A}$ by replacing its $i$th row with the
row ${\bf b}$.
 Denote the $j$th column  and the $i$th row of  $
{\bf A}^{*} $ by $ {\bf a}_{\,\textbf{.}j}^{*} $ and $ {\bf
a}_{i\,\textbf{.}}^{*} $, respectively.
\begin{thm} \cite{kyr4}\label{theor:det_repr_MP}
If $ {\bf A} \in  {\mathbb{H}}_{r}^{m\times n} $, then
the Moore-Penrose inverse  $ {\bf A}^{ \dag} = \left( {a_{ij}^{
\dag} } \right) \in  {\mathbb{H}}^{n\times m} $ have the
following determinantal representations
  \begin{equation}
\label{eq:det_repr_A*A}
 a_{ij}^{ \dag}  = {\frac{{{\sum\limits_{\beta
\in J_{r,n} {\left\{ {i} \right\}}} {{\rm{cdet}} _{i} \left(
{\left( { {\bf A}^{ *}  {\bf A}} \right)_{. i}
\left( { {\bf a}_{.j}^{ *} }  \right)} \right)
 _{\beta} ^{\beta} } } }}{{{\sum\limits_{\beta \in
J_{r,n}} {{\left|   {\bf A}^{ *}  {\bf A}
\right| _{\beta} ^{\beta}  }}} }}},
\end{equation}
or
  \begin{equation}
\label{eq:det_repr_AA*} a_{ij}^{ \dag}  =
{\frac{{{\sum\limits_{\alpha \in I_{r,m} {\left\{ {j} \right\}}}
{{\rm{rdet}} _{j} \left( {( {\bf A} {\bf A}^{ *}
)_{j.} ( {\bf a}_{i.}^{ *} )} \right)_{\alpha}
^{\alpha} } }}}{{{\sum\limits_{\alpha \in I_{r,m}}  {{
{\left|  {\bf A}{\bf A}^{ *}  \right|
_{\alpha} ^{\alpha} } }}} }}}.
\end{equation}
 \end{thm}
\begin{rem}\label{rem:unit_repr}Note that   for an arbitrary full-rank matrix, $ {\bf A} \in  {\mathbb{H}}_{r}^{m\times n} $, a column-vector ${\bf d}_{.j}$, and a row-vector ${\bf d}_{i.}$ with appropriate sizes,  we put, respectively,
if $ r=n$, then
  \begin{gather*}{\rm{cdet}} _{i} \left(
{\left( { {\bf A}^{ *} {\bf A}} \right)_{. i}
\left( {\bf d}_{.j}  \right)} \right)=
\sum\limits_{\beta
\in J_{n,n} {\left\{ {i} \right\}}} {{\rm{cdet}} _{i} \left(
{\left( { {\bf A}^{ *}  {\bf A}} \right)_{. i}
\left( {\bf d}_{.j}  \right)} \right)
 _{\beta} ^{\beta} },\\\det\left(  { {\bf A}^{ *}  {\bf A}}
\right)={\sum\limits_{\beta \in
J_{n,n}} {{\left|  { {\bf A}^{ *} {\bf A}}
\right|_{\beta} ^{\beta}  }}}; \end{gather*}
if $r=m$, then
 \begin{gather*} {\rm{rdet}} _{j} \left( {( {\bf A} {\bf A}^{ *}
)_{j.} \left( {\bf d}_{i.}  \right)} \right)=
{\sum\limits_{\alpha \in I_{m,m} {\left\{ {j} \right\}}}
{{\rm{rdet}} _{j} \left( {( {\bf A} {\bf A}^{ *}
)_{j.} \left( {\bf d}_{i.}  \right)} \right)\,_{\alpha}
^{\alpha} } },\\\det\left( { {\bf A} {\bf A}^{ *} } \right)={\sum\limits_{\alpha \in I_{m,m}}  {{
{\left| { {\bf A} {\bf A}^{ *} } \right|
_{\alpha} ^{\alpha} } }}}.
 \end{gather*}

\end{rem}
\begin{cor}\label{cor:det_repr_proj_P}
If $ {\bf A} \in {\mathbb{H}}_{r}^{m\times n} $, then the
projection matrix $ {\bf A}^{\dag}  {\bf A} = : {\bf
P}_A = \left( {p_{ij}} \right)_{n\times n} $ have the
determinantal representation
  \begin{equation}
\label{eq:det_repr_proj_P}
p_{ij} = {\frac{{{\sum\limits_{\beta \in J_{r,n} {\left\{ {i}
\right\}}} {{\rm{cdet}} _{i} \left( {\left( { {\bf A}^{ *}
 {\bf A}} \right)_{\,\textbf{.}\,i} \left({\bf \dot{a}}_{\,\textbf{.}j} \right)}
\right)  _{\beta} ^{\beta} } }}}{{{\sum\limits_{\beta
\in J_{r,n}}  {{ {\left| { {\bf A}^{ *}  {\bf
A}} \right|  _{\beta} ^{\beta} }}}} }}},
 \end{equation}
\noindent where $ {\bf \dot{a}}_{\,\textbf{.}j} $  denotes  the $j$th column of
${ {\bf A}^{ *}  {\bf A}} \in
{\mathbb{H}}^{n\times n}$.
\end{cor}
\begin{cor}\label{cor:det_repr_proj_Q}
If $ {\bf A} \in  {\mathbb{H}}_{r}^{m\times n}$, then the
projection matrix $ {\bf A} {\bf A}^{\dag} = : {\bf
Q}_A = \left( {q_{ij}} \right)_{m\times m} $ have the
determinantal representation
  \begin{equation}
\label{eq:det_repr_proj_Q}
q_{ij} = {\frac{{{\sum\limits_{\alpha \in I_{r,m} {\left\{ {j}
\right\}}} {{{\rm{rdet}} _{j} {\left( {( {\bf A} {\bf A}^{ *}
)_{j \textbf{.}}\, ({\bf \ddot{a}}  _{i  \textbf{.}} )}
\right)  _{\alpha} ^{\alpha} } }}}
}}{{{\sum\limits_{\alpha \in I_{r,m}} {{{\left| {
{\bf A} {\bf A}^{ *} } \right| _{\alpha
}^{\alpha} }  }}} }}},
 \end{equation}
\noindent where ${\bf \ddot{a}} _{i\textbf{.}} $ denotes the $i$th row of $
{\bf A} {\bf A}^{*}\in {\mathbb{H}}^{m\times m}$.
\end{cor}
The following important orthogonal projectors ${\bf L}_{A}:={\bf I}-{\bf A}^{\dag}{\bf A}$ and ${\bf R}_{A}:={\bf I}-{\bf A}{\bf A}^{\dag}$ induced from ${\bf A}$ will be used below.
\begin{thm}\cite{wang04} \label{theor:LS_AXB}
Let ${\bf A} \in {\mathbb{H}}^{m\times n}$, ${\bf B}\in  {\mathbb{H}}^{r\times s}$, ${\bf C} \in {\mathbb{H}}^{m\times s}$ be known and ${\bf X} \in {\mathbb{H}}^{n\times r}$ be unknown. Then
the matrix equation \begin{equation}\label{eq:AXB} {\bf A}{\bf X} {\bf B} =  {\bf
C}\end{equation} is consistent if and only if $ {\bf A}{\bf A}^{\dag}{\bf C}{\bf B}^{\dag}{\bf B}={\bf C}$.
 In this case, its general
solution can be expressed as
\begin{equation}\label{eq:sol_AXB} {\bf X} =  {\bf A}^{\dag}{\bf
C} {\bf B}^{\dag} +  {\bf L}_{A} {\bf V}+ {\bf W} {\bf R}_{B},\end{equation}
where $ {\bf V}, {\bf W}$    are arbitrary matrices over ${\mathbb{H}}$ with appropriate dimensions.
\end{thm}
\begin{thm}\cite{kyr5}\label{theor:AXB=D}
Let ${\bf A} \in {\mathbb{H}}_{r_{1}}^{m\times n}$, ${\bf B} \in {\mathbb{H}}_{r_{2}}^{r\times s}$. Then the partial solution $ {\bf X}^{0}= {\bf A}^{\dag} {\bf C} {\bf B}^{\dag}=(x^{0}_{ij})\in
{\mathbb{H}}^{n\times r}$ to (\ref{eq:AXB} )  have determinantal representations,
\begin{equation*}
x^0_{ij} = {\frac{{{\sum\limits_{\beta \in J_{r_{1},n} {\left\{
{i} \right\}}} {{\rm{cdet}} _{i} \left( {\left( { {\bf A}^{
*} {\bf A}} \right)_{.i} \left( {{{\bf
d}}_{.j}^{ B}} \right)} \right) _{\beta} ^{\beta}
} } }}{{{\sum\limits_{\beta \in J_{r_{1},n}} {{\left| {\bf A}^{ *}  {\bf A}
\right|_{\beta} ^{\beta}}} \sum\limits_{\alpha \in I_{r_{2},r}}{{\left| {\bf B} {\bf B}^{ *}
\right|_{\alpha} ^{\alpha}}}} }}},
\end{equation*}
or
\begin{equation*}
 x^0_{ij}={\frac{{{\sum\limits_{\alpha
\in I_{r_{2},r} {\left\{ {j} \right\}}} {{\rm{rdet}} _{j} \left(
{\left( {{\bf B}{\bf B}^{ *} } \right)_{j.} \left(
{{ {\bf d}}_{i.}^{ A}} \right)}
\right)_{\alpha} ^{\alpha} } }}}{{{\sum\limits_{\beta \in
J_{r_{1},n}} {{\left|  {\bf A}^{ *}  {\bf A}   \right|_{\beta} ^{\beta}}}\sum\limits_{\alpha \in
I_{r_{2},r}} {{\left|  {\bf B}{\bf B}^{ *}  \right| _{\alpha} ^{\alpha}}}} }}},
\end{equation*}
where
\begin{gather*}
   {{{\bf d}}_{.j}^{  B}}=\left[
\sum\limits_{\alpha \in I_{r_{2},r} {\left\{ {j} \right\}}}
{{\rm{rdet}} _{j} \left( {\left( { {\bf B}{\bf B}^{ *}
} \right)_{j.} \left( {\tilde{ {\bf c}}_{k.}} \right)}
\right)_{\alpha} ^{\alpha}}
\right]\in {\mathbb{H}}^{n \times
1},\,\,\,\,k=1,\ldots,n, \\  {{ {\bf d}}_{i.}^{ A}}=\left[
\sum\limits_{\beta \in J_{r_{1},n} {\left\{ {i} \right\}}}
{{\rm{cdet}} _{i} \left( {\left(  {\bf A}^{ *} {\bf A}
 \right)_{.i} \left( {\tilde{ {\bf C}}_{.l}} \right)}
\right)_{\beta} ^{\beta}} \right]\in {\mathbb{H}}^{1 \times
r},\,\,\,\,l=1,\ldots,r,
\end{gather*}
 are the column vector and the row vector, respectively.  ${\tilde{ {\bf
 c}}_{i.}}$ and
${\tilde{ {\bf c}}_{.j}}$ are the $i$th row   and the $j$th
column  of $ {\bf \widetilde{C}}=  {\bf
A}^\ast{\bf C}{\bf B}^\ast$.
\end{thm}
\begin{cor}\label{cor:sol_AX} Let ${\bf A} \in {\mathbb{H}}_{k}^{m\times n}$, ${\bf C} \in {\mathbb{H}}^{m\times s}$ be known and ${\bf X} \in {\mathbb{H}}^{n\times s}$ be unknown. Then the matrix equation $ {\bf A}{\bf X}  =  {\bf
C}$ is consistent if and only if $ {\bf A}{\bf A}^{\dag}{\bf C}={\bf C}$.
 In this case, its general
solution can be expressed as
$ {\bf X} =  {\bf A}^{\dag}{\bf
C} +  {\bf L}_{A} {\bf V}$,
where $ {\bf V}$    is an arbitrary matrix over ${\mathbb{H}}$ with appropriate dimensions. The partial solution $ {\bf X}^0={\bf A}^{\dag}{\bf C}$ has the following determinantal representation,
\begin{equation*}
 x^0_{ij} = {\frac{{{\sum\limits_{\beta \in
J_{k,n} {\left\{ {i} \right\}}} {{\rm{cdet}} _{i} \left( {\left(
{ {\bf A}^{ *}  {\bf A}} \right)_{.i} \left(
{\hat{{\bf c}}_{.j}} \right)} \right)
_{\beta} ^{\beta} } } }}{{{\sum\limits_{\beta \in J_{k,n}}
{{\left| { {\bf A}^{ *}  {\bf A}}   \right|_{\beta} ^{\beta}}}} }}}.
\end{equation*}
\noindent where $\hat{ {\bf c}}_{.j}$ is the $j$th column of
$\hat{ {\bf C}}={\bf A}^{ \ast}{\bf C}$.
\end{cor}
\begin{cor} \label{cor:sol_XB}
 Let ${\bf B}\in  {\mathbb{H}}_k^{r\times s}$, ${\bf C}\in  {\mathbb{H}}^{n\times s}$ be given, and ${\bf X}\in  {\mathbb{H}}^{n\times r}$ be unknown.
 Then
the equation ${\bf X}{\bf B}= {\bf C}$
is solvable if and only if ${\bf C} ={\bf C}{\bf B}^{\dag}{\bf B}$ and its general solution is
$ {\bf X} = {\bf C}{\bf B}^{\dag} + {\bf W}{\bf R}_{B}$,
where ${\bf W}$ is a any matrix with conformable dimension.
Moreover, its partial solution  $
{\bf X}={\bf C} {\bf B}^{\dag}$
 has the  determinantal representation,
\begin{equation*}
 x_{ij} = {\frac{{{\sum\limits_{\alpha \in I_{k,r}
{\left\{ {j} \right\}}} {{\rm{rdet}} _{j} \left( {\left( {{\bf B}{\bf B}^{ *} } \right)_{j.} \left( {\hat{
{\bf c}}_{i.}} \right)} \right)_{\alpha} ^{\alpha} }
}}}{{{\sum\limits_{\alpha \in I_{k,r}}  {{ {\left| {{\bf B}{\bf B}^{ *} } \right| _{\alpha}
^{\alpha} } }}} }}}.
\end{equation*}
\noindent where ${\hat{{\bf
c}}_{i.}} $ is the $i$th row of
$\hat{ {\bf C}}={\bf C} {\bf B}^{
\ast}$.
\end{cor}

\section{Determinantal representations of a partial solution to the generalized Sylvestr equation (\ref{eq:tss1}).}
\begin{lem}\cite{wang04}\label{lem:cond1} Let ${\bf A}_{1}\in  {\mathbb{H}}^{m\times n}$, ${\bf B}_{1}\in  {\mathbb{H}}^{r\times s}$,  ${\bf A}_{2}\in  {\mathbb{H}}^{m\times p}$, ${\bf B}_{2}\in  {\mathbb{H}}^{q\times s}$,  ${\bf C}\in  {\mathbb{H}}^{m\times s}$.
 Put ${\bf M} = {\bf R}_{A_1}{\bf A}_{2}$, ${\bf N} ={\bf B}_{2}{\bf L}_{B_1} $, ${\bf S} ={\bf A}_{2}{\bf L}_{M} $. Then
the following results are equivalent.

\begin{itemize}
  \item[(i)] Eq. (\ref{eq:tss1}) has a  solution $({\bf X}_{1}, {\bf X}_{2})$, where ${\bf X}_{1}\in  {\mathbb{H}}^{n\times r}$, ${\bf X}_{2}\in  {\mathbb{H}}^{p\times q}$.
  \item [(ii)]\begin{align}\label{eq:cond1}{\bf R}_{M}{\bf R}_{A_1}{\bf C}={\bf 0},~{\bf R}_{A_1}{\bf C}{\bf L}_{B_2}={\bf 0},~{\bf C}{\bf L}_{B_2}{\bf L}_{N}={\bf 0},~{\bf R}_{A_2}{\bf C}{\bf L}_{B_1}={\bf 0}.\end{align}
  \item [(iii)]\begin{align*}{\bf Q}_{M}{\bf R}_{A_1}{\bf C}{\bf P}_{B_2}={\bf R}_{A_1}{\bf C},~{\bf Q}_{A_2}{\bf C}{\bf L}_{B_1}{\bf P}_{N}={\bf C}{\bf L}_{B_1}.\end{align*}
     \item [(iv)]
       $\rank\left[{\bf A}_{1}\,{\bf A}_{2}\,{\bf C}\right]=\rank\left[{\bf A}_{1}\,{\bf A}_{2}\right]$,  $\rank\left[{\bf B}_{1}^*\,{\bf B}_{2}^*\,{\bf C}^*\right]=\rank\left[{\bf B}_{1}^*\,{\bf B}_{2}^*\right]$,\newline $\rank\begin{bmatrix}{\bf A}_{1}&{\bf C}\\{\bf 0}&{\bf B}_{2}\end{bmatrix}= \rank\begin{bmatrix}{\bf A}_{1}&{\bf 0}\\{\bf 0}&{\bf B}_{2}\end{bmatrix}$, $\rank\begin{bmatrix}{\bf A}_{2}&{\bf C}\\{\bf 0}&{\bf B}_{1}\end{bmatrix}= \rank\begin{bmatrix}{\bf A}_{2}&{\bf 0}\\{\bf 0}&{\bf B}_{1}\end{bmatrix}$.
\end{itemize}
In  that case,  the general solution  of (\ref{eq:tss1}) can  be expressed  as the following,
\begin{multline}\label{eq:tss1_sol1}
{\bf X}_{1}={\bf A}^{\dag}_{1}{\bf C} {\bf B}^{\dag}_{1}-
{\bf A}^{\dag}_{1}{\bf A}_{2}{\bf M}^{\dag}{\bf R}_{A_1}{\bf C} {\bf B}^{\dag}_{1}-
{\bf A}^{\dag}_{1}{\bf S}{\bf A}^{\dag}_{2}{\bf C}{\bf L}_{B_1}{\bf N}^{\dag}{\bf B}_{2}{\bf B}^{\dag}_{1}-\\
{\bf A}^{\dag}_{1}{\bf S}{\bf V}{\bf R}_{N}{\bf B}_{2}{\bf B}^{\dag}_{1}+{\bf L}_{A_1}{\bf U}+{\bf Z}{\bf R}_{B_1},
\end{multline}
\begin{multline}\label{eq:tss1_sol2}
{\bf X}_{2}={\bf M}^{\dag}{\bf R}_{A_1}{\bf C}{\bf B}_{2}^{\dag}+
{\bf L}_{M}{\bf S}^{\dag}{\bf S}{\bf A}^{\dag}_{2}{\bf C}{\bf L}_{B_1}{\bf N}^{\dag}+
{\bf L}_{M}({\bf V}-{\bf S}^{\dag}{\bf S}{\bf V}{\bf N}{\bf N}^{\dag})+{\bf W}{\bf R}_{B_2},
\end{multline}
where ${\bf U}$, ${\bf V}$, ${\bf Z}$ and ${\bf W}$ are arbitrary matrices of suitable shapes over  ${\mathbb H}$.
\end{lem}
Some simplifications of (\ref{eq:tss1_sol1}) and (\ref{eq:tss1_sol2}) can be derived due to the quaternionic analogues of the following propositions.
\begin{lem}\cite{mac} If ${\bf A}\in {\mathbb{H}}^{n\times n}$ is Hermitian and idempotent, then the following equation holds for any matrix ${\bf B}\in {\mathbb{H}}^{m\times n}$,
\begin{align}\label{eq:reverse_order_l}
{\bf A}({\bf B}{\bf A})^{\dag}=&({\bf B}{\bf A})^{\dag},\\
\label{eq:reverse_order_r}
({\bf A}{\bf B})^{\dag}{\bf A}=&({\bf A}{\bf B})^{\dag}.
\end{align}
\end{lem}
Since ${\bf R}_{A_{1}}$, ${\bf L}_{B_{1}}$, and ${\bf L}_{M}$ are  projectors, then by (\ref{eq:reverse_order_l}) and (\ref{eq:reverse_order_r}), we have, respectively,
 \begin{multline}\label{eq:sympl_x0}
{\bf M}^{\dag}{\bf R}_{A_{1}}=({\bf R}_{A_{1}}{\bf A}_{2})^{\dag}{\bf R}_{A_{1}}=({\bf R}_{A_{1}}{\bf A}_{2})^{\dag}={\bf M}^{\dag},\\
{\bf L}_{B_1}{\bf N}^{\dag} ={\bf L}_{B_1}\left( {\bf B}_{2}{\bf L}_{B_1}\right)^{\dag} =\left( {\bf B}_{2}{\bf L}_{B_1}\right)^{\dag}={\bf N}^{\dag},\\
{\bf L}_{M}{\bf S}^{\dag} ={\bf L}_{M}\left( {\bf A}_{2}{\bf L}_{M}\right)^{\dag} =\left( {\bf A}_{2}{\bf L}_{M}\right)^{\dag}={\bf S}^{\dag}.
\end{multline}
Using (\ref{eq:sympl_x0}), we obtain the following expressions  of  (\ref{eq:tss1_sol1}) and (\ref{eq:tss1_sol2}),
\begin{multline*}
{\bf X}_{1}={\bf A}^{\dag}_{1}{\bf C} {\bf B}^{\dag}_{1}-
{\bf A}^{\dag}_{1}{\bf A}_{2}{\bf M}^{\dag}{\bf C} {\bf B}^{\dag}_{1}-
{\bf A}^{\dag}_{1}{\bf S}{\bf A}^{\dag}_{2}{\bf C}{\bf N}^{\dag}{\bf B}_{2}{\bf B}^{\dag}_{1}-
{\bf A}^{\dag}_{1}{\bf S}{\bf V}{\bf R}_{N}{\bf B}_{2}{\bf B}^{\dag}_{1}+\\{\bf L}_{A_1}{\bf U}+{\bf Z}{\bf R}_{B_1},\\
{\bf X}_{2}={\bf M}^{\dag}{\bf C}{\bf B}_{2}^{\dag}+
{\bf P}_{S}{\bf A}^{\dag}_{2}{\bf C}{\bf N}^{\dag}+
{\bf L}_{M}({\bf V}-{\bf P}_{S}{\bf V}{\bf Q}_{N})+{\bf W}{\bf R}_{B_2}.
\end{multline*}
By putting ${\bf U}$, ${\bf V}$, ${\bf Z}$, and ${\bf W}$ as zero-matrices of suitable shapes, we obtain the following partial solution to (\ref{eq:tss1}),
\begin{gather}\label{eq:part_sol_x1}
{\bf X}_{1}={\bf A}^{\dag}_{1}{\bf C} {\bf B}^{\dag}_{1}-
{\bf A}^{\dag}_{1}{\bf A}_{2}{\bf M}^{\dag}{\bf C} {\bf B}^{\dag}_{1}-
{\bf A}^{\dag}_{1}{\bf S}{\bf A}^{\dag}_{2}{\bf C}{\bf N}^{\dag}{\bf B}_{2}{\bf B}^{\dag}_{1},\\
\label{eq:part_sol_x2}{\bf X}_{2}={\bf M}^{\dag}{\bf C}{\bf B}_{2}^{\dag}+
{\bf P}_{S}{\bf A}^{\dag}_{2}{\bf C}{\bf N}^{\dag}.
\end{gather}
Further we give determinantal representations of (\ref{eq:part_sol_x1})-(\ref{eq:part_sol_x2}). Let ${\bf A}_{1}\in  {\mathbb{H}}^{m\times n}_{r_{1}}$, ${\bf B}_{1}\in  {\mathbb{H}}^{r\times s}_{r_{2}}$, ${\bf A}_{2}\in  {\mathbb{H}}^{m\times p}_{r_{3}}$, ${\bf B}_{2}\in  {\mathbb{H}}^{q\times s}_{r_{4}}$,     $\rank {\bf M} =r_{5}$, $\rank {\bf N} =r_{6}$, and $\rank {\bf S}=r_{7}$.

First, consider each term of (\ref{eq:part_sol_x1}) separately.

(i) By Theorem \ref{theor:AXB=D} for the first term of (\ref{eq:part_sol_x1}), ${\bf A}^{\dag}_{1}{\bf C} {\bf B}^{\dag}_{1}:={\bf X}_{11}=\left(x_{ij}^{(11)}\right)$,  we have
\begin{equation}\label{eq:x111}
x_{ij}^{(11)} = {\frac{{{\sum\limits_{\beta \in J_{r_{1},n} {\left\{
{i} \right\}}} {{\rm{cdet}} _{i} \left( {\left( {{\bf A}_{1}^{
*}  {\bf A}_{1}} \right)_{.i} \left( { {\bf
d}_{.j}^{ B_1}} \right)} \right) _{\beta} ^{\beta}
} } }}{{{\sum\limits_{\beta \in J_{r_{1},n}} {{\left| {{\bf A}_{1}^{
*}  {\bf A}_{1}} \right|_{\beta} ^{\beta} }} \sum\limits_{\alpha \in I_{r_{2},r}}{{\left|
{ {\bf B}_{1} {\bf B}_{1}^{ *} } \right| _{\alpha} ^{\alpha} }}} }}},
\end{equation}
or
\begin{equation}\label{eq:x112}
 x_{ij}^{(11)}={\frac{{{\sum\limits_{\alpha
\in I_{r_{2},r} {\left\{ {j} \right\}}} {{\rm{rdet}} _{j} \left(
{\left(
{ {\bf B}_{1} {\bf B}_{1}^{ *} } \right)_{j.} \left(
{{ {\bf d}}_{i.}^{{ A}_1}} \right)}
\right)_{\alpha} ^{\alpha} } }}}{{{\sum\limits_{\beta \in
J_{r_{1},n}} {{\left| {{\bf A}_{1}^{
*}  {\bf A}_{1}} \right| _{\beta} ^{\beta} }}\sum\limits_{\alpha \in
I_{r_{2},r}} {{\left|(
{ {\bf B}_{1} {\bf B}_{1}^{ *} } \right| _{\alpha} ^{\alpha} }}} }}},
\end{equation}
where
\begin{align*}
   {{{\bf d}}_{.\,j}^{{ B}_1}}=&\left[
\sum\limits_{\alpha \in I_{r_{2},r} {\left\{ {j} \right\}}}
{{\rm{rdet}} _{j} \left( {\left(
{ {\bf B}_{1} {\bf B}_{1}^{ *} } \right)_{j.} \left( {\bf{ c}}_{k.}^{(1)}\right)}
\right)_{\alpha} ^{\alpha}}
\right]\in {\mathbb{H}}^{n \times
1},\,\,\,\,k=1,\ldots,n,\\
  {{ {\bf d}}_{i\,.}^{ { A}_1}}=&\left[
\sum\limits_{\beta \in J_{r_{1},n} {\left\{ {i} \right\}}}
{{\rm{cdet}} _{i} \left( {\left( {{\bf A}_{1}^{
*}  {\bf A}_{1}} \right)_{.i} \left( {\bf{ c}}_{.\,l}^{(1)}\right)}
\right)_{\beta} ^{\beta}} \right]\in{\rm {\mathbb{H}}}^{1 \times
r},\,\,\,\,l=1,\ldots,r,
\end{align*}
 are the column vector and the row  vector, respectively.  ${ {\bf
 c}}^{(1)}_{k.}$ and
${{\bf c}^{(1)}_{.\,l}}$ are the $k$th row   and the $l$th
column  of $ { {\bf C}}_{1}:={\bf A}_{1}^{*}{\bf C}{\bf B}_{1}^{*}$.

(ii) Using the determinantal representation (\ref{eq:det_repr_A*A}) for  ${\bf A}_1^{\dag}$ and by Theorem \ref{theor:AXB=D}, we obtain the the following representation of the second term of (\ref{eq:part_sol_x1}), ${\bf A}^{\dag}_{1}{\bf A}_{2}{\bf M}^{\dag}{\bf C} {\bf B}^{\dag}_{1}:={\bf X}_{12}=\left(x_{ij}^{(12)}\right)$,
  \begin{equation}\label{eq:x12}
x_{ij}^{(12)} = {\frac{\sum\limits_{t=1}^p
\sum\limits_{\beta
\in J_{r_{1},n} {\left\{ {i} \right\}}} {{\rm{cdet}} _{i} \left(
{\left(
{ {\bf A}_1^{ *} {\bf A}_1 } \right)_{.i} \left(
{{ \tilde{\bf a}}_{.\,t}^{( 2)}} \right)}
\right)_{\beta} ^{\beta}\varphi_{tj}
 }
 }{{{\sum\limits_{\beta \in
J_{r_{1},n}} {{\left| {{\bf A}_1^{ *} {\bf A}_1} \right| _{\beta} ^{\beta} }}\sum\limits_{\beta \in
J_{r_{5},p}} {{\left|
{ {\bf M}^{ *}  {\bf M}} \right| _{\beta} ^{\beta} }}\sum\limits_{\alpha \in I_{r_{2},r}}{{\left|
{ {\bf B}_1 {\bf B}_1^{ *} } \right| _{\alpha} ^{\alpha} }}
} }}},
\end{equation}
where
\begin{equation*}
\varphi_{tj} ={{\sum\limits_{\beta \in J_{r_{5},p} {\left\{
{t} \right\}}} {{\rm{cdet}} _{t} \left( {\left( {{\bf M}^{
*}  {\bf M}} \right)_{.\,t} \left( {{ {\bf
\psi}}_{.\,j}^ {B_1}} \right)} \right) _{\beta} ^{\beta}
} } },
\end{equation*}
or
\begin{equation*}
\varphi_{tj}  ={{{\sum\limits_{\alpha
\in I_{r_{2},r} {\left\{ {j} \right\}}} {{\rm{rdet}} _{j} \left(
{\left(
{ {\bf B}_1 {\bf B}_1^{ *} } \right)_{j.} \left(
{{ {\bf \psi}}_{t.}^{ M}} \right)}
\right)_{\alpha} ^{\alpha} } }}},
\end{equation*}
and
\begin{align*}
   { {\bf
\psi}}_{.j}^ {B_1}=&\left[
\sum\limits_{\alpha \in I_{r_{2},r} {\left\{ {f} \right\}}}
{{\rm{rdet}}_{j} \left( {\left(
{ {\bf B}_1 {\bf B}_1^{ *} } \right)_{j.} \left( {\bf c}^{( 2)}_{k.}\right)}
\right)_{\alpha} ^{\alpha}}
\right]\in{\rm {\mathbb{H}}}^{p \times
1},\,\,\,\,k=1,\ldots,p, \\
 { {\bf \psi}}_{t\,.}^{ M}=&\left[
\sum\limits_{\beta \in J_{r_{5},p} {\left\{ {t} \right\}}}
{{\rm{cdet}} _{t} \left( {\left( {{\bf M}^{
*}  {\bf M}} \right)_{.\,t} \left( {\bf c}^{( 2)}_{.\,l}\right)}
\right)_{\beta} ^{\beta}} \right]\in {\mathbb{H}}^{1 \times
r},\,\,\,\,l=1,\ldots,r,
\end{align*}
 are the column vector and the row  vector, respectively.
$ { \tilde{\bf a}}_{.\,t}^{( 2)}$ is the $t$th column of $ {\tilde  {\bf A}}_2:={\bf A}_1^{*}{\bf A}_{2}$,
$ {\bf
 c}_{k.}^{( 2)}$ and
${\bf c}_{.\,l}^{( 2)}$ are the $k$th row   and the $l$th
column  of ${\bf C}_2:={\bf M}^{*}{\bf C}{\bf B}_1^{*}$, respectively.

(iii) For the third term of (\ref{eq:part_sol_x1}), ${\bf A}^{\dag}_{1}{\bf S}{\bf A}^{\dag}_{2}{\bf C}{\bf N}^{\dag}{\bf B}_{2}{\bf B}^{\dag}_{1}:={\bf X}_{13}=\left(x_{ij}^{(13)}\right)$, we use the determinantal representations (\ref{eq:det_repr_A*A}) to   ${\bf A}_1^{\dag}$ and  (\ref{eq:det_repr_AA*}) to  ${\bf B}_1^{\dag}$, respectively. Then, due to Theorem \ref{theor:AXB=D} for ${\bf A}_2^{\dag}{\bf C}{\bf N}^{\dag}$, we have
\footnotesize
      \begin{multline}\label{eq:x13}
x_{ij}^{(13)} = \\{\frac{\sum\limits_{t=1}^p\sum\limits_{f=1}^q\sum\limits_{\beta
\in J_{r_{1},n} {\left\{ {i} \right\}}} {{\rm{cdet}}_{i} \left(
{\left(
{ {\bf A}_1^{ *} {\bf A}_1 } \right)_{.i} \left(
{{ \tilde{\bf s}}_{.\,t}} \right)}
\right)_{\beta} ^{\beta} }\,\eta_{tf}
\sum\limits_{\alpha
\in I_{r_{2},r} {\left\{ {j} \right\}}} {{\rm{rdet}}_{j} \left(
{\left(
{ {\bf B}_1 {\bf B}_1^{ *} } \right)_{j.} \left(
{{ \tilde{\bf b}}_{f.}^{( 2)}} \right)}
\right)_{\alpha} ^{\alpha}
 }
 }{{{\sum\limits_{\beta \in
J_{r_{1},n}} { {\left|
{ {\bf A}_1^{ *}  {\bf A}_1} \right| _{\beta} ^{\beta} } }
\sum\limits_{\beta \in
J_{r_{3},p}} { {\left|
{ {\bf A}_2^{ *}  {\bf A}_2} \right| _{\beta} ^{\beta} } }
\sum\limits_{\alpha \in
I_{r_{6},q}} { {\left| {{\bf N}  {\bf N}^{
*}} \right|_{\alpha} ^{\alpha} } }
\sum\limits_{\alpha \in
I_{r_{2},r}} { {\left| {{\bf B}_1  {\bf B}_1^{
*}} \right|_{\alpha} ^{\alpha} } }
} }}},
\end{multline}\normalsize
where
\begin{equation}\label{eq:x131}
\eta_{tf} ={{\sum\limits_{\beta \in J_{r_{3},p} {\left\{
{t} \right\}}} {{\rm{cdet}} _{t} \left( {\left( {{\bf A}_2^{
*}  {\bf A}_2} \right)_{.\,t} \left( {{ {\bf
\zeta}}_{.\,f}^ {N}} \right)} \right) _{\beta} ^{\beta}
} } },
\end{equation}
or
\begin{equation}\label{eq:x132}
\eta_{tf} ={{{\sum\limits_{\alpha
\in I_{r_{6},q} {\left\{ {f} \right\}}} {{\rm{rdet}} _{f} \left(
{\left(
{ {\bf N} {\bf N}^{ *} } \right)_{f.} \left(
{{ {\bf {\zeta}}}_{t.}^{ A_2}} \right)}
\right)_{\alpha} ^{\alpha} } }}},
\end{equation}
and
\begin{align*}
   {{ {\bf
\zeta}}_{.\,f}^ {N}}=&\left[
\sum\limits_{\alpha \in I_{r_{6},q} {\left\{ {f} \right\}}}
{{\rm{rdet}} _{f} \left( {\left(
{ {\bf N} {\bf N}^{ *} } \right)_{f.} \left( {\bf c}^{( 3)}_{k.}\right)}
\right)_{\alpha} ^{\alpha}}
\right]\in{\rm {\mathbb{H}}}^{p \times
1},\,\,\,\,k=1,\ldots,p,
\\  {{ {\bf {\zeta}}}_{t.}^{ A_2}}=&\left[
\sum\limits_{\beta \in J_{r_{3},p} {\left\{ {t} \right\}}}
{{\rm{cdet}} _{t} \left( {\left( {{\bf A}_2^{
*}  {\bf A}_2} \right)_{.\,t} \left( {\bf c}^{( 3)}_{.\,l}\right)}
\right)_{\beta} ^{\beta}} \right]\in {\mathbb{H}}^{1 \times
q},\,\,\,\,l=1,\ldots,q,
\end{align*}
 are the column vector and the row  vector, respectively.
${{ \tilde{\bf s}}_{.\,t}}$ is the $t$th column of ${ \tilde{\bf S}}:={\bf A}_1^{*}{\bf S}$,
${{ \tilde{\bf b}}_{f.}^{( 2)}}$ is the $f$th row of ${ \tilde{\bf B}}_2:={\bf B}_2{\bf B}_1^{*}$,
 ${\bf c}^{( 3)}_{k.}$ and ${\bf c}^{( 3)}_{.\,l}$ are the $k$th row and  the $f$th  column of ${\bf C}_3:={\bf A}_2^{*}{\bf C}{\bf N}^{*}$.

Now, we consider each term of (\ref{eq:part_sol_x2}).

(i) Due to Theorem \ref{theor:AXB=D} for the first term ${\bf M}^{\dag}{\bf C}{\bf B}_{2}^{\dag}=:{\bf X}_{21}=(x_{gf}^{(21)})$ of (\ref{eq:part_sol_x2}), we have
\begin{equation}\label{eq:x211}
x_{gf}^{(21)} = \frac{{{\sum\limits_{\beta \in J_{r_{5},p} {\left\{
{g} \right\}}} {{\rm{cdet}} _{g} \left( {\left( {{\bf M}^{
*}  {\bf M}} \right)_{.g} \left( { {\bf
d}_{.f}^{ B_2}} \right)} \right) _{\beta} ^{\beta}
} } }}{{{{\sum\limits_{\beta \in J_{r_{5},p}} {{\left|  {{\bf M}^{
*}  {\bf M}}
\right|_{\beta} ^{\beta}}} \sum\limits_{\alpha \in I_{r_{4},q}}{{\left|
{ {\bf B}_{2} {\bf B}_{2}^{ *} } \right| _{\alpha} ^{\alpha} }}} }}},
\end{equation}
or
\begin{equation}\label{eq:x212}
 x_{gf}^{(21)}={\frac{{{\sum\limits_{\alpha
\in I_{r_{4},q} {\left\{ {f} \right\}}} {{\rm{rdet}} _{f} \left(
{\left(
{ {\bf B}_{2} {\bf B}_{2}^{ *} } \right)_{f.} \left(
{{ {\bf d}}_{g.}^{M}} \right)}
\right)_{\alpha} ^{\alpha} } }}}{{{{\sum\limits_{\beta \in J_{r_{5},p}} {{\left|  {{\bf M}^{
*}  {\bf M}}
\right|_{\beta} ^{\beta}}} \sum\limits_{\alpha \in I_{r_{4},q}}{{\left|
{ {\bf B}_{2} {\bf B}_{2}^{ *} } \right| _{\alpha} ^{\alpha} }}} }}}},
\end{equation}
where
\begin{gather*}
   {{{\bf d}}_{.\,f}^{{ B}_2}}=\left[
\sum\limits_{\alpha \in I_{r_{4},q} {\left\{ {f} \right\}}}
{{\rm{rdet}} _{f} \left( {\left(
{ {\bf B}_{2} {\bf B}_{2}^{ *} } \right)_{f.} \left( {\bf{ c}}_{k.}^{(4)}\right)}
\right)_{\alpha} ^{\alpha}}
\right]\in {\mathbb{H}}^{p \times
1},\,\,\,\,k=1,\ldots,p,\\  {{ {\bf d}}_{g\,.}^{ M}}=\left[
\sum\limits_{\beta \in J_{r_{5},p} {\left\{ {g} \right\}}}
{{\rm{cdet}} _{g} \left( {\left( {{\bf M}^{
*}  {\bf M}} \right)_{.g} \left( {\bf{ c}}_{.\,l}^{(4)}\right)}
\right)_{\beta} ^{\beta}} \right]\in{\mathbb{H}}^{1 \times
q},\,\,\,\,l=1,\ldots,q,
\end{gather*}
 are the column vector and the row  vector, respectively.  ${ {\bf
 c}}^{(4)}_{k.}$ and
${{\bf c}^{(4)}_{.l}}$ are the $k$th row   and the $l$th
column  of $ { {\bf C}}_{4}:={\bf M}^{*}{\bf C}{\bf B}_{2}^{*}$.

(ii) Finally, for the second term ${\bf P}_{S}{\bf A}^{\dag}_{2}{\bf C}{\bf N}^{\dag}=:{\bf X}_{22}=(x_{gf}^{(22)})$ of (\ref{eq:part_sol_x2})
using (\ref{eq:det_repr_proj_P}) for a determinantal representation of   ${\bf P}_S$, and due to Theorem \ref{theor:AXB=D} for ${\bf A}^{\dag}_{2}{\bf C}{\bf N}^{\dag}$, we obtain
      \begin{equation}\label{eq:x22}
x_{gf}^{(22)} = \\{\frac{\sum\limits_{t=1}^p\sum\limits_{\beta
\in J_{r_{7},p} {\left\{ {g} \right\}}} {{\rm{cdet}} _{g} \left(
{\left(
{ {\bf S}^{ *} {\bf S} } \right)_{.g} \left(
{{ \ddot{\bf s}}_{.t}} \right)}
\right)_{\beta} ^{\beta} }\,\eta_{tf}
 }{{{\sum\limits_{\beta \in
J_{r_{7},p}} { {\left|
{ {\bf S}^{ *}  {\bf S}} \right| _{\beta} ^{\beta} } }
\sum\limits_{\beta \in
J_{r_{3},p}} { {\left|
{ {\bf A}_2^{ *}  {\bf A}_2} \right| _{\beta} ^{\beta} } }
\sum\limits_{\alpha \in
I_{r_{6},q}} { {\left| {{\bf N}  {\bf N}^{
*}} \right| _{\alpha} ^{\alpha} } }
} }}},
\end{equation}
where $\eta_{tf}$ is (\ref{eq:x131}) or (\ref{eq:x132}).

So, we prove the following theorem.
\begin{thm}\label{th:sol_1}Let ${\bf A}_{1}\in  {\mathbb{H}}^{m\times n}_{r_{1}}$, ${\bf B}_{1}\in  {\mathbb{H}}^{r\times s}_{r_{2}}$, ${\bf A}_{2}\in  {\mathbb{H}}^{m\times p}_{r_{3}}$, ${\bf B}_{2})\in  {\mathbb{H}}^{q\times s}_{r_{4}}$,   $\rank{\bf M}=r_{5}$, $\rank {\bf N} =r_{6}$, ${\rank {\bf S}}=r_{7}$. Then the pair solution (\ref{eq:part_sol_x1})-(\ref{eq:part_sol_x2}),
${\bf X}_{1}=\left(x^{(1)}_{ij}\right)\in  {\mathbb{H}}^{n\times r}$, ${\bf X}_{2}=\left(x^{(2)}_{gf}\right)\in  {\mathbb{H}}^{p\times q}$ to Eq.  (\ref{eq:tss1})
by the components
\begin{equation*}x^{(1)}_{ij}=x^{(11)}_{ij}-x^{(12)}_{ij}-x^{(13)}_{ij}, ~~~~~
x^{(2)}_{gf}=x^{(21)}_{gf}+x^{(22)}_{gf},
\end{equation*}
has the determinantal representation,
where the term $x^{(11)}_{ij}$ is (\ref{eq:x111}) or (\ref{eq:x112}), $x^{(12)}_{ij}$ is (\ref{eq:x12}),  $x^{(13)}_{ij}$ is (\ref{eq:x13}),  $x^{(21)}_{gf}$ is (\ref{eq:x211}) or (\ref{eq:x212}),  $x^{(22)}_{gf}$ is (\ref{eq:x22}).
\end{thm}

\section{Cramer's Rules for special cases of  (\ref{eq:tss1}) with only one two-sided term.}

In this section, we consider all special cases of  (\ref{eq:tss1}) when only second term  is two-sided.

1. Let  in Eq.(\ref{eq:tss1})  the matrix ${\bf B}_{1}$ be vanish, i.e. ${\bf B}_{1}={\bf I}_s$.  Then, we have the equation
\begin{equation}\label{eq:tss11}
 {\bf A}_{1}{\bf X}_1+ {\bf A}_{2}{\bf X}_2{\bf B}_{2}={\bf C},
\end{equation}
where ${\bf A}_{1}\in  {\mathbb{H}}^{m\times n}$,  ${\bf A}_{2}\in  {\mathbb{H}}^{m\times p}$, ${\bf B}_{2}\in  {\mathbb{H}}^{q\times s}$,  ${\bf C}\in  {\mathbb{H}}^{m\times s}$ be given,
 ${\bf X}_{1}\in  {\mathbb{H}}^{n\times s}$ and ${\bf X}_{2}\in  {\mathbb{H}}^{p\times q}$ are to be determined. Since ${\bf L}_{B_1}={\bf R}_{B_1}={\bf 0}$, ${\bf N}={\bf B}_2{\bf L}_{B_1}={\bf 0}$, and  ${\bf L}_{N}={\bf R}_{N}={\bf I}$ and taking into account the simplifications by (\ref{eq:reverse_order_l}) and (\ref{eq:reverse_order_r}), then we derive the following analog of Lemma \ref{lem:cond1}.
\begin{lem}\label{lm:sol_11} Let  ${\bf M} = {\bf R}_{A_1}{\bf A}_{2}$,  ${\bf S} ={\bf A}_{2}{\bf L}_{M} $. Then
the following results are equivalent.
\begin{itemize}
  \item[(i)] Eq. (\ref{eq:tss11}) is solvable.
  \item [(ii)]${\bf R}_{M}{\bf R}_{A_1}{\bf C}={\bf 0}$,  ${\bf R}_{A_1}{\bf C}{\bf L}_{B_2}={\bf 0}$.
  \item [(iii)]${\bf Q}_{M}{\bf R}_{A_1}{\bf C}{\bf P}_{B_2}={\bf R}_{A_1}{\bf C}$.
     \item [(iv)]
       $\rank\left[{\bf A}_{1}\,{\bf A}_{2}\,{\bf C}\right]=\rank\left[{\bf A}_{1}\,{\bf A}_{2}\right]$,  $\rank\begin{bmatrix}{\bf A}_{1}&{\bf C}\\{\bf 0}&{\bf B}_{2}\end{bmatrix}= \rank\begin{bmatrix}{\bf A}_{1}&{\bf 0}\\{\bf 0}&{\bf B}_{2}\end{bmatrix}$.
\end{itemize}
In  that case,  the general solution  of (\ref{eq:tss1}) can  be expressed  as  follows,
\begin{align*}
{\bf X}_{1}=&{\bf A}^{\dag}_{1}{\bf C} -
{\bf A}^{\dag}_{1}{\bf A}_{2}{\bf M}^{\dag}{\bf C} -
{\bf A}^{\dag}_{1}{\bf S}{\bf V}{\bf B}_{2}+{\bf L}_{A_1}{\bf U},
\\
{\bf X}_{2}=&{\bf M}^{\dag}{\bf C}{\bf B}_{2}^{\dag}+
{\bf L}_{M}{\bf V}+{\bf W}{\bf R}_{B_2},
\end{align*}
where ${\bf U}$, ${\bf V}$, and ${\bf W}$ are arbitrary matrices of suitable shapes over  ${\mathbb H}$ .
\end{lem}
By putting ${\bf U}$, ${\bf V}$, and ${\bf W}$ as zero-matrices of suitable shapes, we obtain the following partial solution of (\ref{eq:tss11}),
\begin{align}\label{eq:part_sol_1x1}
{\bf X}_{1}=&{\bf A}^{\dag}_{1}{\bf C} -
{\bf A}^{\dag}_{1}{\bf A}_{2}{\bf M}^{\dag}{\bf C},\\
\label{eq:part_sol_1x2}{\bf X}_{2}=&{\bf M}^{\dag}{\bf C}{\bf B}_{2}^{\dag}.
\end{align}
Further we give determinantal representations of (\ref{eq:part_sol_1x1})-(\ref{eq:part_sol_1x2}).
\begin{thm}\label{th:sol_11}Let ${\bf A}_{1}\in  {\mathbb{H}}^{m\times n}_{r_{1}}$,  ${\bf A}_{2}\in  {\mathbb{H}}^{m\times p}_{r_{2}}$, ${\bf B}_{2}\in  {\mathbb{H}}^{q\times s}_{r_{3}}$,  and $\rank {\bf M}={\rm min}\{\rank {\bf A}_{2}, \rank{\bf R}_{A_{1}}\} =r_{4}$. Then the solution ${\bf X}_{1}=\left(x^{(1)}_{ij}\right)\in  {\mathbb{H}}^{n\times s}$ from (\ref{eq:part_sol_1x1})  has the  determinantal representation
\small
\begin{multline}\label{eq:1x1}
x_{ij}^{(1)} = {\frac{
\sum\limits_{\beta
\in J_{r_{1},n} {\left\{ {i} \right\}}} {{\rm{cdet}}_{i} \left(
{\left(
{ {\bf A}_1^{ *} {\bf A}_1 } \right)_{.i} \left(
{{ {\bf c}}\,_{.j}^{( 1)}} \right)}
\right)_{\beta} ^{\beta}
 }
 }{{{\sum\limits_{\beta \in
J_{r_{1},n}} {{\left| {{\bf A}_1^{ *} {\bf A}_1} \right|_{\beta}^{\beta} }}
} }}}-\\ {\frac{\sum\limits_{t=1}^p
\sum\limits_{\beta
\in J_{r_{1},n} {\left\{ {i} \right\}}} {{\rm{cdet}}_{i} \left(
{\left(
{ {\bf A}_1^{ *} {\bf A}_1 } \right)_{.i} \left(
{{ \tilde{\bf a}}_{.\,t}^{( 2)}} \right)}
\right)_{\beta} ^{\beta}\sum\limits_{\beta \in J_{r_{4},p} {\left\{
{t} \right\}}} {{\rm{cdet}}_{t} \left( {\left( {{\bf M}^{
*}  {\bf M}} \right)_{.\,t} \left( { {\bf c}}_{.\,j}^{( 2)} \right)} \right)_{\beta}^{\beta}
} } }{{{\sum\limits_{\beta \in
J_{r_{1},n}} {{\left| {{\bf A}_1^{ *} {\bf A}_1} \right|_{\beta}^{\beta} }}\sum\limits_{\beta \in
J_{r_{4},p}} {{\left|
{ {\bf M}^*  {\bf M}} \right|_{\beta}^{\beta} }}
} }}},
\end{multline}\normalsize
where $ { \tilde{\bf a}}_{.\,t}^{( 2)}$ is the $t$th column of $ {\tilde  {\bf A}}_2:={\bf A}_1^{*}{\bf A}_{2}$,
$ {\bf
 c}_{.j}^{( 1)}$ and
${\bf c}_{.j}^{( 2)}$ are  the $j$th
columns  of $   {\bf C}_1:={\bf A}_1^{*}{\bf C}$ and  $   {\bf C}_2:={\bf M}^{*}{\bf C}$, respectively. The  solution ${\bf X}_{2}=\left(x^{(2)}_{gf}\right)\in  {\mathbb{H}}^{p\times q}$ from (\ref{eq:part_sol_1x2}) has the  determinantal representations
\begin{equation}\label{eq:1x21}
x_{gf}^{(2)} = \frac{{{\sum\limits_{\beta \in J_{r_{4},p} {\left\{
{g} \right\}}} {{\rm{cdet}} _{g} \left( {\left( {{\bf M}^{
*}  {\bf M}} \right)_{.g} \left( { {\bf
d}_{.f}^{ B_2}} \right)} \right) _{\beta} ^{\beta}
} } }}{{{{\sum\limits_{\beta \in J_{r_{4},p}} {{\left|  {{\bf M}^{
*}  {\bf M}}
\right|_{\beta} ^{\beta}}} \sum\limits_{\alpha \in I_{r_{3},q}}{{\left|
{ {\bf B}_{2} {\bf B}_{2}^{ *} } \right| _{\alpha} ^{\alpha} }}} }}},
\end{equation}
or
\begin{equation}\label{eq:1x22}
 x_{gf}^{(2)}={\frac{{{\sum\limits_{\alpha
\in I_{r_{3},q} {\left\{ {f} \right\}}} {{\rm{rdet}} _{f} \left(
{\left(
{ {\bf B}_{2} {\bf B}_{2}^{ *} } \right)_{f.} \left(
{{ {\bf d}}_{g.}^{M}} \right)}
\right)_{\alpha} ^{\alpha} } }}}{{{{\sum\limits_{\beta \in J_{r_{4},p}} {{\left|  {{\bf M}^{
*}  {\bf M}}
\right|_{\beta} ^{\beta}}} \sum\limits_{\alpha \in I_{r_{3},q}}{{\left|
{ {\bf B}_{2} {\bf B}_{2}^{ *} } \right| _{\alpha} ^{\alpha} }}} }}}},
\end{equation}
where
\begin{align*}
   {{{\bf d}}_{.\,f}^{{ B}_2}}=&\left[
\sum\limits_{\alpha \in I_{r_{3},q} {\left\{ {f} \right\}}}
{{\rm{rdet}} _{f} \left( {\left(
{ {\bf B}_{2} {\bf B}_{2}^{ *} } \right)_{f.} \left( {\bf{ c}}_{k.}^{(3)}\right)}
\right)_{\alpha} ^{\alpha}}
\right]\in {\mathbb{H}}^{p \times
1},\,\,\,\,k=1,\ldots,p,\\
  {{ {\bf d}}_{g\,.}^{ M}}=&\left[
\sum\limits_{\beta \in J_{r_{4},p} {\left\{ {g} \right\}}}
{{\rm{cdet}} _{g} \left( {\left( {{\bf M}^{
*}  {\bf M}} \right)_{.g} \left( {\bf{ c}}_{.\,l}^{(3)}\right)}
\right)_{\beta} ^{\beta}} \right]\in{\mathbb{H}}^{1 \times
q},\,\,\,\,l=1,\ldots,q,
\end{align*}
 are the column vector and the row  vector, respectively.  $ {\bf
 c}^{(3)}_{k.}$ and
${\bf c}^{(3)}_{.\,l}$ are the $k$th row   and the $l$th
column  of $ { {\bf C}}_{3}:={\bf M}^{*}{\bf C}{\bf B}_{2}^{*}$.

\end{thm}
\begin{proof} By using Corollary \ref{cor:sol_AX} to the both terms of (\ref{eq:part_sol_1x1}) and Theorem  \ref{theor:AXB=D} to (\ref{eq:part_sol_1x2}), we evidently obtain the determinantal representations (\ref{eq:1x1}) and (\ref{eq:1x21})-(\ref{eq:1x22}), respectively.

\end{proof}

2. Let now  in Eq.(\ref{eq:tss1})  the matrix ${\bf A}_{1}$ be vanish, i.e. ${\bf A}_{1}={\bf I}_m. $ Then we have the equation
\begin{equation}\label{eq:tss12}
 {\bf X}_1{\bf B}_{1}+ {\bf A}_{2}{\bf X}_2{\bf B}_{2}={\bf C},
\end{equation}
where ${\bf B}_{1}\in  {\mathbb{H}}^{r\times s}$,  ${\bf A}_{2}\in  {\mathbb{H}}^{m\times p}$, ${\bf B}_{2}\in  {\mathbb{H}}^{q\times s}$,  ${\bf C}\in  {\mathbb{H}}^{m\times s}$ be given,
 ${\bf X}_{1}\in  {\mathbb{H}}^{m\times r}$ and ${\bf X}_{2}\in  {\mathbb{H}}^{p\times q}$ are to be determined. Since ${\bf L}_{A_1}={\bf 0}$, ${\bf R}_{A_1}={\bf 0}$, ${\bf M}={\bf A}_2{\bf L}_{A_1}={\bf 0}$,  ${\bf L}_{M}={\bf I}$, and ${\bf S} ={\bf A}_{2}{\bf L}_{M}={\bf A}_{2} $ and taking into account the simplifications by (\ref{eq:reverse_order_l}) and (\ref{eq:reverse_order_r}), then we derive the following lemma similar to Lemma \ref{lem:cond1}.
\begin{lem} Let  ${\bf N} ={\bf B}_{2}{\bf L}_{B_1} $. Then
the following results are equivalent.
\begin{itemize}
  \item[(i)] Eq. (\ref{eq:tss12}) is solvable.
  \item [(ii)] ${\bf C}{\bf L}_{B_2}{\bf L}_{N}={\bf 0}$, ${\bf R}_{A_2}{\bf C}{\bf L}_{B_1}={\bf 0}$.
  \item [(iii)] ${\bf Q}_{A_2}{\bf C}{\bf L}_{B_1}{\bf P}_{N}={\bf C}{\bf L}_{B_1}$.
     \item [(iv)]
     $\rank\left[{\bf B}_{1}^*\,{\bf B}_{2}^*\,{\bf C}^*\right]=\rank\left[{\bf B}_{1}^*\,{\bf B}_{2}\right]$, $\rank\begin{bmatrix}{\bf A}_{2}&{\bf C}\\{\bf 0}&{\bf B}_{1}\end{bmatrix}= \rank\begin{bmatrix}{\bf A}_{2}&{\bf 0}\\&{\bf B}_{1}\end{bmatrix}$.
\end{itemize}
In  that case,  the general solution  of (\ref{eq:tss12}) can  be expressed  as follows
\begin{align*}
{\bf X}_{1}&={\bf C} {\bf B}^{\dag}_{1}-
{\bf Q}_{A_2}{\bf C}{\bf N}^{\dag}{\bf B}_{2}{\bf B}^{\dag}_{1}-
{\bf A}_{2}{\bf V}{\bf R}_{N}{\bf B}_{2}{\bf B}^{\dag}_{1}+{\bf Z}{\bf R}_{B_1},
\\
{\bf X}_{2}&={\bf A}^{\dag}_{2}{\bf C}{\bf N}^{\dag}+{\bf V}-
{\bf P}_{A_2}{\bf V}{\bf Q}_N+{\bf W}{\bf R}_{B_2},
\end{align*}
where  ${\bf V}$, ${\bf Z}$ and ${\bf W}$ are arbitrary matrices of suitable shapes over  ${\mathbb H}$ .
\end{lem}
By putting  ${\bf V}$, ${\bf Z}$, and ${\bf W}$ as zero-matrices of suitable shapes, we obtain the following partial solution of (\ref{eq:tss12}),
\begin{align}\label{eq:part_sol_2x1}
{\bf X}_{1}&={\bf C} {\bf B}^{\dag}_{1}-
{\bf Q}_{A_2}{\bf C}{\bf N}^{\dag}{\bf B}_{2}{\bf B}^{\dag}_{1},\\
\label{eq:part_sol_2x2}{\bf X}_{2}&={\bf A}^{\dag}_{2}{\bf C}{\bf N}^{\dag}.
\end{align}
Further we give determinantal representations of (\ref{eq:part_sol_2x1})-(\ref{eq:part_sol_2x2}).
\begin{thm}\label{th:sol_12}Let ${\bf B}_{1}\in  {\mathbb{H}}^{r\times s}_{r_{1}}$,  ${\bf A}_{2}\in  {\mathbb{H}}^{m\times p}_{r_{2}}$, ${\bf B}_{2}\in  {\mathbb{H}}^{q\times s}_{r_{3}}$, and  $\rank {\bf N} =r_{4}$. Then  the  solution (\ref{eq:part_sol_2x1})  has the  determinantal representation
\begin{multline}\label{eq:2x1}
x_{ij}^{(1)} = {\frac{
\sum\limits_{\alpha
\in I_{r_{1},r} {\left\{ {j} \right\}}} {{\rm{rdet}} _{j} \left(
{\left(
{ {\bf B}_1 {\bf B}_1^{ *} } \right)_{j.} \left(
{{ {\bf c}}\,_{i.}^{( 1)}} \right)}
\right)_{\alpha} ^{\alpha}
 }
 }{{{\sum\limits_{\alpha
\in I_{r_{1},r}} {{\left| {\bf B}_1 {\bf B}_1^{ *} \right| _{\alpha} ^{\alpha} }}
} }}}-\\
{\frac{\sum\limits_{l=1}^{m}\sum\limits_{t=1}^{q}
x^{(11)}_{il}x^{(12)}_{lt}x^{(13)}_{tj}
 }{{{\sum\limits_{\alpha \in I_{r_2,m}} {{{\left| {
{\bf A}_2 {\bf A}_2^{ *} } \right| _{\alpha
}^{\alpha} }  }}}{\sum\limits_{\alpha \in
I_{r_{4},q}} {{\left| {\bf N} {\bf N}^{ *} \right| _{\alpha} ^{\alpha} }}\sum\limits_{\alpha\in
I_{r_{1},r}} {{\left|
{ {\bf B}_1 {\bf B}_1^{ *}} \right| _{\alpha} ^{\alpha} }}
} }}},
\end{multline}
where
\begin{align*}
x^{(11)}_{il}&=
{\sum\limits_{\alpha \in I_{r_2,m} {\left\{ {l}
\right\}}} {{{\rm{rdet}} _{l} {\left( {( {\bf A}_2 {\bf A}_2^{ *}
)_{l .}\, ({\bf \ddot{a}}^{(2)}  _{i .} )}
\right)  _{\alpha} ^{\alpha} } }}},\\
x^{(12)}_{lt}&=
{\sum\limits_{\alpha \in I_{r_4,q} {\left\{ {t}
\right\}}} {{{\rm{rdet}} _{t} {\left( {( {\bf N} {\bf N}^{ *}
)_{t.}\, ({\bf {c}}^{(2)}  _{l.} )}
\right)  _{\alpha} ^{\alpha} } }}},\\
x^{(13)}_{tj}&=
\sum\limits_{\alpha
\in I_{r_{1},r} {\left\{ {j} \right\}}} {{\rm{rdet}} _{j} \left(
{\left(
{ {\bf B}_1 {\bf B}_1^{ *} } \right)_{j.} \left(
{{ \tilde{{\bf b}}}_{t.}^{( 2)}} \right)}
\right)_{\alpha} ^{\alpha}
 },
\end{align*}
and $ {\bf \ddot{a}}_{i.}^{(2)}$ is the $i$th row of $ {\bf A}_2{\bf A}_{2}^{*}$,
${\bf
 c}_{i.}^{( 1)}$ and
${\bf c}_{l.}^{( 2)}$ are  the $i$th and $l$th
rows  of $   {\bf C}_1:={\bf C}{\bf B}_1^{*}$ and  $   {\bf C}_2:={\bf C}{\bf N}^{*}$, respectively, and ${ \tilde{\bf b}}_{t.}^{( 2)}$ is the $t$th row of $   {\tilde{\bf B}}_2:={\bf B}_2{\bf B}_1^{*}$. The  solution (\ref{eq:part_sol_2x2})  has the  determinantal representations
\begin{equation}\label{eq:2x21}
x_{gf}^{(2)} = \frac{{{\sum\limits_{\alpha \in I_{r_4,q} {\left\{ {f}
\right\}}} {{{\rm{rdet}} _{f} {\left( {( {\bf N} {\bf N}^{ *}
)_{f.}\, ({{{\bf d}}_{g.}^{{ A}_2}} )}
\right)  _{\alpha} ^{\alpha} } }}} }}{{{{\sum\limits_{\beta \in J_{r_{2},p}}{{\left|
{ {\bf A}_{2}^{ *} {\bf A}_{2} } \right| _{\beta} ^{\beta} }}\sum\limits_{\alpha \in I_{r_4,q}} {{\left|  {{\bf N}  {\bf N}^{
*}}
\right|_{\alpha} ^{\alpha}}|} }
}}},
\end{equation}
or
\begin{equation}\label{eq:2x22}
 x_{gf}^{(2)}={\frac{{{\sum\limits_{\beta
\in J_{r_{2},p} {\left\{ {g} \right\}}} {{\rm{cdet}} _{g} \left(
{\left(
{ {\bf A}_{2}^{ *} {\bf A}_{2} } \right)_{.g} \left(
{{ {\bf d}}_{.f}^{N}} \right)}
\right)_{\alpha} ^{\alpha} } }}}{{{{\sum\limits_{\beta \in J_{r_{2},p}}{{\left|
{ {\bf A}_{2}^{ *} {\bf A}_{2} } \right| _{\beta} ^{\beta} }}\sum\limits_{\alpha \in I_{r_4,q}} {{\left|  {{\bf N}  {\bf N}^{
*}}
\right|_{\alpha} ^{\alpha}}|} }  }}}},
\end{equation}
where
\begin{align*}{{ {\bf d}}_{g.}^{ A_2}}=&\left[
\sum\limits_{\beta \in J_{r_{2},p} {\left\{ {g} \right\}}}
{{\rm{cdet}} _{g} \left( {\left( {{\bf A}_2^{
*}  {\bf A}_2} \right)_{.g} \left( {\bf{ c}}_{.\,l}^{(3)}\right)}
\right)_{\beta} ^{\beta}} \right]\in{\mathbb{H}}^{1 \times
q},\,\,\,\,l=1,\ldots,q,
   \\ {{{\bf d}}_{.f}^{N}}=&\left[
\sum\limits_{\alpha \in I_{r_{4},q} {\left\{ {f} \right\}}}
{{\rm{rdet}} _{f} \left( {\left(
{ {\bf N} {\bf N}^{ *} } \right)_{f.} \left( {\bf{ c}}_{k.}^{(3)}\right)}
\right)_{\alpha} ^{\alpha}}
\right]\in {\mathbb{H}}^{p \times
1},\,\,\,\,k=1,\ldots,p,
\end{align*}
 are the  row  vector and the column vector, respectively. ${{\bf c}^{(3)}_{.\,l}}$ and ${ {\bf
 c}}^{(3)}_{k.}$
 are the $l$th
column  and  the $k$th row  of $ { {\bf C}}_{3}:={\bf A}_{2}^{*}{\bf C}{\bf N}^{*}$.

\end{thm}
\begin{proof} Using Corollary \ref{cor:sol_XB} to the first term of (\ref{eq:part_sol_2x1}) and to the multipliers $x^{(12)}_{lt}$ and $x^{(13)}_{tj}$, and  Corollary  \ref{cor:det_repr_proj_Q} to the multiplier $x^{(11)}_{il}$ of the second term of (\ref{eq:part_sol_2x1}), we evidently obtain the determinantal representation (\ref{eq:2x1}). Using Theorem \ref{theor:AXB=D} to  (\ref{eq:part_sol_2x2}), we similarly get  (\ref{eq:2x21})-(\ref{eq:2x22}).

\end{proof}

3. Finally, consider the case when both matrices  ${\bf A}_{1}$ and   ${\bf B}_{1}$  are vanish  in Eq.(\ref{eq:tss1}), i.e. ${\bf A}_{1}={\bf I}_m$ and ${\bf B}_{1}={\bf I}_r$. Then we have the equation
\begin{equation}\label{eq:tss13}
 {\bf X}_1+ {\bf A}_{2}{\bf X}_2{\bf B}_{2}={\bf C},
\end{equation}
where   ${\bf A}_{2}\in  {\mathbb{H}}^{m\times p}$,  ${\bf B}_{2}\in  {\mathbb{H}}^{q\times r}$,  ${\bf C}\in  {\mathbb{H}}^{m\times r}$ be given,
 ${\bf X}_{1}\in  {\mathbb{H}}^{m\times r}$ and ${\bf X}_{2}\in  {\mathbb{H}}^{p\times q}$ are to be determined.

  This equation is the famous generalized Stein equation.

  Since ${\bf L}_{A_1}={\bf 0}$, ${\bf R}_{A_1}={\bf 0}$, ${\bf M}={\bf A}_2{\bf L}_{A_1}={\bf 0}$,  ${\bf L}_{M}={\bf I}$, and ${\bf S} ={\bf A}_{2}{\bf L}_{M}={\bf A}_{2} $, ${\bf L}_{B_1}={\bf R}_{B_1}={\bf 0}$, ${\bf N}={\bf B}_2{\bf L}_{B_1}={\bf 0}$, and  ${\bf L}_{N}={\bf R}_{N}={\bf I}$  and taking account the simplifications by (\ref{eq:reverse_order_l}) and (\ref{eq:reverse_order_r}), then we have the following  lemma.
\begin{lem}The following results are equivalent.
\begin{itemize}
  \item[(i)] Eq. (\ref{eq:tss13}) is solvable.
 \item [(ii)]${\bf C}{\bf L}_{B_2}={\bf 0} $.
     \item [(iii)]
       $\rank\begin{bmatrix}{\bf A}_{2}&{\bf C}\end{bmatrix}=\rank[{\bf A}_{2}]$,  $\rank\begin{bmatrix}{\bf B}_{2}^*&{\bf C}^*\end{bmatrix}=\rank[{\bf B}_{2}^*]$.
\end{itemize}
In  that case,  the general solution  of (\ref{eq:tss1}) can  be expressed  as the following,
\begin{equation}\label{eq:tss13_sol1}
{\bf X}_{1}={\bf C}
-
{\bf A}_{2}{\bf V}{\bf B}_{2},
\end{equation}
\begin{equation}\label{eq:tss13_sol2}
{\bf X}_{2}={\bf V}+{\bf W}{\bf R}_{B_2},
\end{equation}
where  ${\bf V}$ and ${\bf W}$ are arbitrary matrices of suitable shapes over  ${\mathbb H}$.
\end{lem}
Since determinantal representations of (\ref{eq:tss13_sol1})-(\ref{eq:tss13_sol2}) are evidently, we omit them.

\section{Cramer's rules for special cases of  (\ref{eq:tss1}) with both one-sided terms.}

In this section, we consider all special cases of Eq.(\ref{eq:tss1}) when its both terms  are one-sided.

1. Let   the matrices ${\bf B}_{1}$ and ${\bf A}_{2}$ be vanish in Eq.(\ref{eq:tss1}), i.e. ${\bf B}_{1}={\bf I}_s $ and ${\bf A}_{2}={\bf I}_m $. Then we have the equation
\begin{equation}\label{eq:tss14}
 {\bf A}_{1}{\bf X}_1+ {\bf X}_2{\bf B}_{2}={\bf C},
\end{equation}
where ${\bf A}_{1}\in  {\mathbb{H}}^{m\times n}$,   ${\bf B}_{2}\in  {\mathbb{H}}^{q\times s}$,  ${\bf C}\in  {\mathbb{H}}^{m\times s}$ be given,
 ${\bf X}_{1}\in  {\mathbb{H}}^{n\times s}$ and ${\bf X}_{2}\in  {\mathbb{H}}^{m\times q}$ are to be determined. This equation is the classical  Sylvester equation.

  So, ${\bf L}_{B_1}={\bf R}_{B_1}={\bf 0}$, ${\bf N}={\bf B}_2{\bf L}_{B_1}={\bf 0}$, ${\bf L}_{A_2}={\bf R}_{A_2}={\bf 0}$, ${\bf P}_{A_2}={\bf I}$,  ${\bf L}_{N}={\bf R}_{N}={\bf I}$, ${\bf M} = {\bf R}_{A_1}$, and ${\bf S} ={\bf L}_{M} $.
Since ${\bf R}_{A_1}$ is  the orthogonal projector onto the kernel of ${\bf A}_1$, the we have \begin{multline}\label{eq:RA1} {\bf A}^{\dag}_{1}{\bf M}^{\dag}= {\bf A}^{\dag}_{1}{\bf R}^{\dag}_{A_1}={\bf A}^{\dag}_{1}({\bf I}-{\bf A}_{1}{\bf A}^{\dag}_{1})^{\dag}={\bf 0},\,
{\bf A}^{\dag}_{1}{\bf L}_{M}={\bf A}^{\dag}_{1}({\bf I}-{\bf R}^{\dag}_{A_1}{\bf R}_{A_1})={\bf A}^{\dag}_{1},\\
{\bf M}^{\dag}{\bf R}_{A_1}={\bf R}^{\dag}_{A_1}{\bf R}_{A_1}={\bf R}_{A_1}.
\end{multline}

Due to (\ref{eq:RA1}) and taking into account of simplifications by (\ref{eq:reverse_order_l}) and (\ref{eq:reverse_order_r}), we have the following analog of Lemma \ref{lem:cond1}.
\begin{lem} The following results are equivalent.
\begin{itemize}
  \item[(i)] Eq. (\ref{eq:tss14}) is solvable.
  \item [(ii)]${\bf R}_{A_1}{\bf C}{\bf L}_{B_2}={\bf 0}.$
     \item [(iii)]
      $\rank\begin{bmatrix}{\bf A}_{1}&{\bf C}\\{\bf 0}&{\bf B}_{2}\end{bmatrix}= \rank\begin{bmatrix}{\bf A}_{1}&{\bf 0}\\{\bf 0}&{\bf B}_{2}\end{bmatrix}.$
\end{itemize}
In  that case,  the general solution  of (\ref{eq:tss14}) can  be expressed  as  follows
\begin{align}\label{eq:tss14_sol1}
{\bf X}_{1}&={\bf A}^{\dag}_{1}{\bf C} -
{\bf A}^{\dag}_{1}{\bf V}{\bf B}_{2}+{\bf L}_{A_1}{\bf U},
\\\label{eq:tss14_sol2}
{\bf X}_{2}&={\bf R}_{A_1}{\bf C}{\bf B}_{2}^{\dag}+
{\bf A}_1{\bf A}^{\dag}_{1}{\bf V}+{\bf W}{\bf R}_{B_2},
\end{align}
where ${\bf U}$, ${\bf V}$,  and ${\bf W}$ are arbitrary matrices of suitable shapes  over ${\mathbb H}$.
\end{lem}
The denoting ${\bf V}_1:={\bf A}^{\dag}_{1}{\bf V}$ in  (\ref{eq:tss14_sol1})-(\ref{eq:tss14_sol2}) gives the expression of the general solution  of (\ref{eq:tss14}) that has been  first derived in \cite{baks}.

By putting ${\bf U}$, ${\bf V}$,  and ${\bf W}$ as zero-matrices of suitable shapes, the following partial solution of  (\ref{eq:tss14}) can be obtained,
\begin{align}\label{eq:part_sol_4x1}
{\bf X}_{1}&={\bf A}^{\dag}_{1}{\bf C} ,\\
\label{eq:part_sol_4x2}{\bf X}_{2}&={\bf C}{\bf B}_{2}^{\dag}-{\bf Q}_{A_1}{\bf C}{\bf B}_{2}^{\dag}.
\end{align}
\begin{thm}\label{th:sol_13}Let ${\bf A}_{1}\in  {\mathbb{H}}^{m\times n}_{r_{1}}$,   ${\bf B}_{2}\in  {\mathbb{H}}^{q\times s}_{r_{2}}$. Then  the  solution (\ref{eq:part_sol_4x1})  has the  determinantal representation
\begin{equation}\label{eq:4x1}
x_{ij}^{(1)} = {\frac{
\sum\limits_{\beta
\in J_{r_{1},n} {\left\{ {i} \right\}}} {{\rm{cdet}} _{i} \left(
{\left(
{ {\bf A}_1^{ *} {\bf A}_1 } \right)_{.i} \left(
{{ {\bf c}}_{.j}^{( 1)}} \right)}
\right)_{\beta} ^{\beta}
 }
 }{{{\sum\limits_{\beta \in
J_{r_{1},n}} {{\left| {{\bf A}_1^{ *} {\bf A}_1} \right| _{\beta} ^{\beta} }}
} }}},
\end{equation}
where
$ {\bf c}_{.j}^{( 1)}$ is  the $j$th
columns  of $   {\bf C}_1:={\bf A}_1^{*}{\bf C}$. The determinantal representation of  (\ref{eq:part_sol_4x2})  is
\begin{multline}\label{eq:4x2}
x_{gf}^{(2)} = {\frac{
\sum\limits_{\alpha
\in I_{r_{2},q} {\left\{ {f} \right\}}} {{\rm{rdet}} _{f} \left(
{\left(
{ {\bf B}_2 {\bf B}_2^{ *} } \right)_{f.} \left(
{{ {\bf c}}_{g.}^{( 2)}} \right)}
\right)_{\alpha} ^{\alpha}
 }
 }{{{\sum\limits_{\alpha \in
 I_{r_{2},q}} {{\left| {{\bf B}_2 {\bf B}^{ *}_2} \right| _{\alpha} ^{\alpha} }}
} }}}-\\
{\frac{\sum\limits_{l=1}^{m}
{\sum\limits_{\alpha \in I_{r_1,m} {\left\{ {l}
\right\}}} {{{\rm{rdet}} _{l} {\left( {( {\bf A}_1 {\bf A}_1^{ *}
)_{l .}\, ({\bf \ddot{a}}^{(1)}  _{g .} )}
\right)  _{\alpha} ^{\alpha} } }}}\sum\limits_{\alpha
\in I_{r_{2},q} {\left\{ {f} \right\}}} {{\rm{rdet}} _{f} \left(
{\left(
{ {\bf B}_2 {\bf B}_2^{ *} } \right)_{f.} \left(
{{ {\bf c}}_{l.}^{( 2)}} \right)}
\right)_{\alpha} ^{\alpha}
 }
 }{{{\sum\limits_{\alpha \in I_{r_1,m}} {{{\left| {
{\bf A}_1 {\bf A}_1^{ *} } \right| _{\alpha
}^{\alpha} }  }}}{\sum\limits_{\alpha\in
I_{r_{2},q}} {{\left|
{ {\bf B}_2 {\bf B}_2^{ *}} \right| _{\alpha} ^{\alpha} }}
} }}},
\end{multline}
where $ {\bf c}_{g.}^{( 2)}$ and ${\bf \ddot{a}}^{(1)}_{g .}$  are  the $g$th
rows  of $   {\bf C}_2:={\bf C}{\bf B}_2^{*}$ and ${\bf A}_1 {\bf A}_1^{ *}$, respectively.
\end{thm}
\begin{proof} Using Corollary \ref{cor:sol_AX} to (\ref{eq:part_sol_4x1}) and Corollaries \ref{cor:sol_XB} and \ref{cor:det_repr_proj_Q} to (\ref{eq:part_sol_4x2}), we evidently obtain the determinantal representations (\ref{eq:4x1}) and (\ref{eq:4x2}), respectively.

\end{proof}

2. Let  in Eq.(\ref{eq:tss1})  the matrices  ${\bf A}_{1}$ and ${\bf B}_{2}$ be vanish, i.e. ${\bf A}_{1}={\bf I}_m $ and ${\bf B}_{2}={\bf I}_q $. Then we have the equation
\begin{equation}\label{eq:tss15}
 {\bf X}_1{\bf B}_{1}+ {\bf A}_{2}{\bf X}_2={\bf C},
\end{equation}
where ${\bf B}_{1}\in  {\mathbb{H}}^{r\times s}$,   ${\bf A}_{2}\in  {\mathbb{H}}^{m\times p}$,  ${\bf C}\in  {\mathbb{H}}^{m\times s}$ be given,
 ${\bf X}_{1}\in  {\mathbb{H}}^{m\times r}$ and ${\bf X}_{2}\in  {\mathbb{H}}^{p\times s}$ are to be determined. So, ${\bf L}_{A_1}={\bf R}_{A_1}={\bf 0}$, ${\bf M}={\bf R}_{A_1}{\bf A}_2={\bf 0}$, ${\bf L}_{B_2}={\bf R}_{B_2}={\bf 0}$, ${\bf P}_{B_2}={\bf I}$,  ${\bf L}_{M}={\bf R}_{M}={\bf I}$, ${\bf N} = {\bf L}_{B_1}$, and ${\bf S} ={\bf A}_{2} $.
Since ${\bf L}_{B_1}$ is  the orthogonal projector onto the kernel of ${\bf B}_1$, then we have \begin{multline}\label{eq:LB1} {\bf N}^{\dag}{\bf B}^{\dag}_{1}= {\bf L}^{\dag}_{B_1}{\bf B}^{\dag}_{1}=({\bf I}-{\bf B}^{\dag}_{1}{\bf B}_{1})^{\dag}{\bf B}^{\dag}_{1}={\bf 0},\,
{\bf R}_{N}{\bf B}^{\dag}_{1}=({\bf I}-{\bf L}_{B_1}{\bf L}^{\dag}_{B_1}){\bf B}^{\dag}_{1}={\bf B}^{\dag}_{1},\\
{\bf L}_{B_1}{\bf N}^{\dag}={\bf L}_{B_1}{\bf L}^{\dag}_{B_1}={\bf L}_{B_1}.
\end{multline}

Due to (\ref{eq:LB1}) and taking into account of simplifications by (\ref{eq:reverse_order_l}) and (\ref{eq:reverse_order_r}), the analog of Lemma \ref{lem:cond1} follows.
\begin{lem} The following results are equivalent.
\begin{itemize}
  \item[(i)] Eq. (\ref{eq:tss15}) is solvable.
  \item [(ii)]${\bf R}_{A_2}{\bf C}{\bf L}_{B_1}={\bf 0}$.
     \item [(iii)]
        $\rank\begin{bmatrix}{\bf A}_{2}&{\bf C}\\{\bf 0}&{\bf B}_{1}\end{bmatrix}= \rank\begin{bmatrix}{\bf A}_{2}&{\bf 0}\\{\bf 0}&{\bf B}_{1}\end{bmatrix}$.
\end{itemize}
In  that case,  the general pair solution  of (\ref{eq:tss15}) is
\begin{align*}
{\bf X}_{1}&={\bf C}{\bf B}^{\dag}_{1} -
{\bf A}_{2}{\bf V}{\bf B}^{\dag}_{1}+{\bf Z}{\bf R}_{B_1},
\\
{\bf X}_{2}&={\bf A}_{2}^{\dag}{\bf C}{\bf L}_{B_1}+{\bf L}_{A_2}{\bf V}{\bf L}_{B_1},
\end{align*}
where  ${\bf V}$ and ${\bf Z}$ are arbitrary matrices over  ${\mathbb H}$ of suitable shapes.
\end{lem}
By putting  ${\bf V}$  and ${\bf Z}$ as zero-matrices of suitable shapes, we have the following partial pair solution of (\ref{eq:tss15}),
\begin{align}\label{eq:part_sol_5x1}
{\bf X}_{1}&={\bf C}{\bf B}^{\dag}_{1} ,\\
\label{eq:part_sol_5x2}{\bf X}_{2}&={\bf A}_{2}^{\dag}{\bf C}-{\bf A}_{2}^{\dag}{\bf C}{\bf P}_{B_1}.
\end{align}
\begin{thm}\label{th:sol_15}Let ${\bf B}_{1}\in  {\mathbb{H}}^{r\times s}_{r_{1}}$,   ${\bf A}_{2}\in  {\mathbb{H}}^{m\times p}_{r_{2}}$. Then   (\ref{eq:part_sol_5x1})  has the  determinantal representation
\begin{equation}\label{eq:5x1}
x_{ij}^{(1)} =
{\frac{
\sum\limits_{\alpha
\in I_{r_{1},r} {\left\{ {j} \right\}}} {{\rm{rdet}} _{j} \left(
{\left(
{ {\bf B}_1 {\bf B}_1^{ *} } \right)_{j.} \left(
{{ {\bf c}}_{i.}^{( 1)}} \right)}
\right)_{\alpha} ^{\alpha}
 }
 }{{{\sum\limits_{\alpha \in
 I_{r_{1},r}} {{\left| {{\bf B}_1 {\bf B}^{ *}_1} \right| _{\alpha} ^{\alpha} }}
} }}}
,
\end{equation}
where
$ {\bf
 c}_{i.}^{( 1)}$ is  the $j$th
row  of $   {\bf C}_1:={\bf C}{\bf B}_1^{*}$. The  solution (\ref{eq:part_sol_5x2})  has the  determinantal representation
\small
\begin{multline}\label{eq:5x2}
x_{gf}^{(2)} = {\frac{
\sum\limits_{\beta
\in J_{r_{2},p} {\left\{ {g} \right\}}} {{\rm{cdet}} _{g} \left(
{\left(
{ {\bf A}_2^{ *} {\bf A}_2 } \right)_{.g} \left(
{{ {\bf c}}_{.f}^{( 2)}} \right)}
\right)_{\beta} ^{\beta}
 }
 }{{{\sum\limits_{\beta
\in J_{r_{2},p}} {{\left| {{\bf A}_2^{ *} {\bf A}_2} \right| _{\beta} ^{\beta} }}
} }}}-\\
{\frac{\sum\limits_{l=1}^{s}\sum\limits_{\beta
\in J_{r_{2},p} {\left\{ {g} \right\}}} {{\rm{cdet}} _{g} \left(
{\left(
{ {\bf A}_2^{ *} {\bf A}_2 } \right)_{.g} \left(
{{ {\bf c}}_{.\,l}^{( 2)}} \right)}
\right)_{\beta} ^{\beta}
 }\sum\limits_{\beta
\in J_{r_{1},s} {\left\{ {l} \right\}}} {{\rm{cdet}} _{l} \left(
{\left(
{ {\bf B}^{ *}_1 {\bf B}_1 } \right)_{.\,l} \left(
{{ \bf{{\dot b}}}_{.\,g}^{( 1)}} \right)}
\right)_{\beta} ^{\beta}
 }
 }{{{\sum\limits_{\beta
\in J_{r_{2},p}} {{{\left| {
{\bf A}_2^{ *} {\bf A}_2 } \right|_{\beta
}^{\beta} }  }}}{\sum\limits_{\beta
\in J_{r_{1},s}} {{\left|
{ {\bf B}_1^{ *} {\bf B}_1} \right|_{\beta} ^{\beta} }}
} }}},
\end{multline}\normalsize
where $ {\bf
 c}_{.f}^{( 2)}$ and ${\bf \dot{b}}^{(1)}  _{.f}$  are  the $f$th
columns  of $   {\bf C}_2:={\bf A}_2^{*}{\bf C}$ and ${\bf B}^{ *}_1 {\bf B}_1$, respectively.
\end{thm}
\begin{proof} Using Corollary \ref{cor:sol_XB} to (\ref{eq:part_sol_5x1}) and Corollaries \ref{cor:sol_AX} and \ref{cor:det_repr_proj_P} to (\ref{eq:part_sol_5x2}), we evidently obtain the determinantal representations (\ref{eq:5x1}) and (\ref{eq:5x2}), respectively.

\end{proof}

3. Let  the matrices  ${\bf B}_{1}$ and ${\bf B}_{2}$ be vanish  in Eq.(\ref{eq:tss1}), i.e. ${\bf B}_{1}={\bf B}_{2}={\bf I}_r $. Then we have the equation
\begin{equation}\label{eq:tss16}
 {\bf A}_{1}{\bf X}_1+ {\bf A}_{2}{\bf X}_2={\bf C},
\end{equation}
where ${\bf A}_{1}\in  {\mathbb{H}}^{m\times n}$,   ${\bf A}_{2}\in  {\mathbb{H}}^{m\times p}$,  ${\bf C}\in  {\mathbb{H}}^{m\times r}$ be given,
 ${\bf X}_{1}\in  {\mathbb{H}}^{n\times r}$ and ${\bf X}_{2}\in  {\mathbb{H}}^{p\times r}$ are to be determined. So, ${\bf L}_{B_1}={\bf R}_{B_1}={\bf 0}$,  ${\bf L}_{B_2}={\bf R}_{B_2}={\bf 0}$, ${\bf P}_{B_2}={\bf I}$,  and ${\bf N} = {\bf 0}$.

Due to  (\ref{eq:reverse_order_l}) and (\ref{eq:reverse_order_r}), the following analog of Lemma \ref{lem:cond1} can be obtained.
\begin{lem}Let  ${\bf M} = {\bf R}_{A_1}{\bf A}_{2}$,  ${\bf S} ={\bf A}_{2}{\bf L}_{M} $. The following results are equivalent.
\begin{itemize}
  \item[(i)] Eq. (\ref{eq:tss16}) is solvable.
  \item [(ii)]${\bf R}_{M}{\bf R}_{A_1}{\bf C}={\bf 0}$.
  \item [(iii)] $\rank\begin{bmatrix}{\bf A}_{1}&{\bf A}_{2}&{\bf C}\end{bmatrix}=\rank\begin{bmatrix}{\bf A}_{1}&{\bf A}_{2}\end{bmatrix}$.
\end{itemize}
In  that case,  the general solution  of (\ref{eq:tss16}) can  be expressed  as  follows
\begin{align*}
{\bf X}_{1}&={\bf A}^{\dag}_{1}{\bf C}-
{\bf A}^{\dag}_{1}{\bf A}_{2}{\bf M}^{\dag}{\bf C}-
{\bf A}^{\dag}_{1}{\bf S}{\bf V}+{\bf L}_{A_1}{\bf U},\\
{\bf X}_{2}&={\bf M}^{\dag}{\bf C}+
{\bf L}_{M}{\bf V}.
\end{align*}
where  ${\bf U}$ and ${\bf V}$ are arbitrary matrices over  ${\mathbb H}$ of suitable shapes.
\end{lem}
By putting  ${\bf U}$ and ${\bf V}$  as zero-matrices of suitable shapes, we have the following partial pair solution of (\ref{eq:tss16}),
\begin{align}\label{eq:part_sol_6x1}
{\bf X}_{1}&={\bf A}^{\dag}_{1}{\bf C}-
{\bf A}^{\dag}_{1}{\bf A}_{2}{\bf M}^{\dag}{\bf C},\\
\label{eq:part_sol_6x2}{\bf X}_{2}&={\bf M}^{\dag}{\bf C}.
\end{align}
The next theorem can be proved similarly to Theorem \ref{th:sol_11}.
\begin{thm}\label{th:sol_16}Let ${\bf A}_{1}\in  {\mathbb{H}}^{m\times n}_{r_{1}}$,   ${\bf A}_{2}\in  {\mathbb{H}}^{m\times p}_{r_{2}}$, and ${\rank} {\bf M}=r_{4}$. Then the determinantal representation of  (\ref{eq:part_sol_6x1})  is the same as (\ref{eq:1x1}), and
 the  solution (\ref{eq:part_sol_6x2})  has the  determinantal representation
\begin{equation*}
x_{gf}^{(2)} = \frac{\sum\limits_{\beta \in J_{r_{4},p} {\left\{
{g} \right\}}} {{\rm{cdet}} _{g} \left( {\left( {{\bf M}^{
*}  {\bf M}} \right)_{.\,g} \left( { {\bf c}}_{.f}^{( 2)} \right)} \right) _{\beta} ^{\beta}
} } {{{\sum\limits_{\beta \in
J_{r_{4},p}} {{\left|
{ {\bf M}^{ *}  {\bf M}} \right| _{\beta} ^{\beta} }}
} }},
\end{equation*}
where $ {\bf
 c}_{.f}^{( 2)}$ is  the $f$th
column  of ${\bf C}_2:={\bf M}^{*}{\bf C}$.
\end{thm}

4. Let, now,    the matrices ${\bf A}_{1}$ and ${\bf A}_{2}$ be vanish in Eq.(\ref{eq:tss1}), i.e. ${\bf A}_{1}={\bf A}_{2}={\bf I}_m$.  Then we have the equation
\begin{equation}\label{eq:tss17}
 {\bf X}_1{\bf B}_{1}+ {\bf X}_2{\bf B}_{2}={\bf C},
\end{equation}
where ${\bf B}_{1}\in  {\mathbb{H}}^{r\times s}$,   ${\bf B}_{2}\in  {\mathbb{H}}^{q\times s}$,  ${\bf C}\in  {\mathbb{H}}^{m\times s}$ be given,
 ${\bf X}_{1}\in  {\mathbb{H}}^{m\times r}$ and ${\bf X}_{2}\in  {\mathbb{H}}^{m\times q}$ are to be determined. Since ${\bf L}_{A_1}={\bf R}_{A_1}={\bf 0}$, ${\bf M}={\bf A}_2{\bf L}_{A_1}={\bf 0}$,  ${\bf L}_{M}={\bf I}$, ${\bf L}_{A_2}={\bf R}_{A_2}={\bf 0}$, ${\bf P}_{A_2}={\bf Q}_{A_2}={\bf I}$,  and ${\bf S} ={\bf A}_{2}{\bf L}_{M}={\bf I} $ and taking into account simplifications by (\ref{eq:reverse_order_l}) and (\ref{eq:reverse_order_r}), then we derive the  analog of Lemma \ref{lem:cond1}.
\begin{lem} Let  ${\bf N} ={\bf B}_{2}{\bf L}_{B_1} $. Then
the following results are equivalent.
\begin{itemize}
  \item[(i)] Eq. (\ref{eq:tss17}) is solvable.
  \item [(ii)] ${\bf C}{\bf L}_{B_2}{\bf L}_{N}={\bf 0}$.
  \item [(iii)]
        $\rank\begin{bmatrix}{\bf B}_{1}^*&{\bf B}_{2}^*&{\bf C}\end{bmatrix}=\rank\begin{bmatrix}{\bf B}_{1}^*&{\bf B}_{2}^*\end{bmatrix}$.
\end{itemize}
In  that case,  the general solution  of (\ref{eq:tss17}) can  be expressed  as  follows
\begin{align*}
{\bf X}_{1}&={\bf C} {\bf B}^{\dag}_{1}-
{\bf C}{\bf N}^{\dag}{\bf B}_{2}{\bf B}^{\dag}_{1}-
{\bf V}{\bf R}_{N}{\bf B}_{2}{\bf B}^{\dag}_{1}+{\bf Z}{\bf R}_{B_1},
\\
{\bf X}_{2}&={\bf C}{\bf N}^{\dag}+{\bf V}{\bf R}_N+{\bf W}{\bf R}_{B_2},
\end{align*}
where  ${\bf V}$, ${\bf Z}$ and ${\bf W}$ are arbitrary matrices over  ${\mathbb H}$ of suitable shapes.
\end{lem}
By putting  ${\bf V}$, ${\bf Z}$, and ${\bf W}$ as zero-matrices of suitable shapes, we obtain the following partial solution to (\ref{eq:tss17}),
\begin{align}\label{eq:part_sol_7x1}
{\bf X}_{1}&={\bf C} {\bf B}^{\dag}_{1}-
{\bf C}{\bf N}^{\dag}{\bf B}_{2}{\bf B}^{\dag}_{1},\\
\label{eq:part_sol_7x2}{\bf X}_{2}&={\bf C}{\bf N}^{\dag}.
\end{align}
\begin{thm}\label{th:sol_72}Let ${\bf B}_{1}\in  {\mathbb{H}}^{r\times s}_{r_{1}}$,  ${\bf B}_{2}\in  {\mathbb{H}}^{q\times s}_{r_{3}}$,  and $\rank {\bf N}=r_{3}$. Then  the pair solution (\ref{eq:part_sol_7x1})-(\ref{eq:part_sol_7x2})  has the  determinantal representation,
\begin{align}\label{eq:7x1}
x_{ij}^{(1)}& = {\frac{
\sum\limits_{\alpha
\in I_{r_{1},r} {\left\{ {j} \right\}}} {{\rm{rdet}} _{j} \left(
{\left(
{ {\bf B}_1 {\bf B}_1^{ *} } \right)_{j.} \left(
{{ {\bf c}}_{i.}^{( 1)}} \right)}
\right)_{\alpha} ^{\alpha}
 }
 }{{{\sum\limits_{\alpha
\in I_{r_{1},r}} {{\left| {\bf B}_1 {\bf B}_1^{ *} \right| _{\alpha} ^{\alpha} }}
} }}}-\\
&{\frac{\sum\limits_{t=1}^{q}
{\sum\limits_{\alpha \in I_{r_3,q} {\left\{ {t}
\right\}}} {{{\rm{rdet}} _{t} {\left( {( {\bf N} {\bf N}^{ *}
)_{t.}\, ({\bf {c}}^{(2)}  _{l.} )}
\right)  _{\alpha} ^{\alpha} } }}}\sum\limits_{\alpha
\in I_{r_{1},r} {\left\{ {j} \right\}}} {{\rm{rdet}} _{j} \left(
{\left(
{ {\bf B}_1 {\bf B}_1^{ *} } \right)_{j.} \left(
{{ \tilde{{\bf b}}}_{t.}^{( 2)}} \right)}
\right)_{\alpha} ^{\alpha}
 }
 }{{{\sum\limits_{\alpha \in
I_{r_{3},q}} {{\left| {\bf N} {\bf N}^{ *} \right| _{\alpha} ^{\alpha} }}\sum\limits_{\alpha\in
I_{r_{1},r}} {{\left|
{ {\bf B}_1 {\bf B}_1^{ *}} \right| _{\alpha} ^{\alpha} }}
} }}},\nonumber\\
x_{gf}^{(2)} &= \frac{{{\sum\limits_{\alpha \in I_{r_3,q} {\left\{ {f}
\right\}}} {{{\rm{rdet}} _{f} {\left( {( {\bf N} {\bf N}^{ *}
)_{f.}\, \left( {\bf{ c}}_{g.}^{(2)}\right)}
\right)  _{\alpha} ^{\alpha} } }}} }}{{{{\sum\limits_{\alpha \in I_{r_3,q}} {{\left|  {{\bf N}  {\bf N}^{
*}}
\right|_{\alpha} ^{\alpha}}} }
}}}.\label{eq:7x2}
\end{align}
where
$ {\bf c}_{i.}^{( 1)}$ and
${\bf c}_{l.}^{( 2)}$ are  the $i$th and $l$th
rows  of $   {\bf C}_1:={\bf C}{\bf B}_1^{*}$ and  $   {\bf C}_2:={\bf C}{\bf N}^{*}$, respectively, and ${ \tilde{\bf b}}_{t.}^{( 2)}$ is the $t$th row of $   {\tilde{\bf B}}_2:={\bf B}_2{\bf B}_1^{*}$.
\end{thm}
\begin{proof} Using Corollary \ref{cor:sol_XB} to (\ref{eq:part_sol_7x2}) and  the both terms of  (\ref{eq:part_sol_7x1}), we evidently obtain the determinantal representations (\ref{eq:7x1}) and  (\ref{eq:7x2}).

\end{proof}

\section{Cramer's rules for  like-Lyapunov equations.}
 The well-known  Lyapunov equation is ${\bf A}{\bf X}+ {\bf X}{\bf A}^{\ast}={\bf B}$. In this section, we consider some like Lyapunov equations.

1. Consider the following matrix equation,
\begin{equation}\label{eq:lyap}
 {\bf A}{\bf X}+ {\bf X}^{\ast}{\bf B}={\bf C},
\end{equation}
where ${\bf A}\in  {\mathbb{H}}^{m\times n}$,   ${\bf B}\in  {\mathbb{H}}^{n\times m}$, and ${\bf C}\in  {\mathbb{H}}^{m\times m}$.
Due to \cite{pia}, the following lemma can be expanded from the complex field to  ${\mathbb{H}}$.
\begin{lem}If Eq. (\ref{eq:lyap}) has solution, and $${\bf A}^{\dag}{\bf C}\left({\bf I}-\frac{1}{2}{\bf P}_B\right)=\left[\left({\bf I}-\frac{1}{2}{\bf Q}_A\right){\bf C}{\bf B}^{\dag}\right]^*,$$ then
\begin{equation}\label{eq:sol_lyap}
 {\bf X}_0={\bf A}^{\dag}{\bf C}\left({\bf I}-\frac{1}{2}{\bf Q}_B\right)
\end{equation}
is a solution to (\ref{eq:lyap}) and $-{\bf R}_A{\bf C}{\bf L}_B={\bf 0}$.
\end{lem}
\begin{proof}The proof is similar to (\cite{pia}, Theorem 1).
\end{proof}

\begin{thm}\label{th:sol_18}Let ${\bf A}\in  {\mathbb{H}}^{m\times n}_{r_{1}}$ and ${\bf B}\in  {\mathbb{H}}^{n\times m}_{r_{2}}$. Then  the  solution  ${\bf X}_0=\left(x_{ij}\right)$ to (\ref{eq:lyap})  has the following determinantal representation,
\begin{equation}\label{eq:8x1}
x_{ij} = {\frac{
\sum\limits_{\beta
\in J_{r_{1},n} {\left\{ {i} \right\}}} {{\rm{cdet}} _{i} \left(
{\left(
{ {\bf A}^{ *} {\bf A} } \right)_{.i} \left(
{{ {\bf c}}_{.j}^{( 1)}} \right)}
\right)_{\beta} ^{\beta}
 }
 }{{{\sum\limits_{\beta
\in J_{r_{1},n}} {{\left| {{\bf A}^{ *} {\bf A}} \right| _{\beta} ^{\beta} }}
} }}}-
\frac{{{\sum\limits_{\alpha \in I_{r_2,n} {\left\{ {j}
\right\}}} {{{\rm{rdet}} _{j} {\left( {( {\bf B} {\bf B}^{ *}
)_{j.} ({{{\bf d}}_{i.}^{{ A}}} )}
\right)_{\alpha} ^{\alpha} } }}} }}{{{{2\sum\limits_{\beta \in J_{r_{1},n}}{{\left|
{ {\bf A}^{ *} {\bf A} } \right| _{\beta} ^{\beta} }}\sum\limits_{\alpha \in I_{r_2,n}} {{\left|  {{\bf B}  {\bf B}^{
*}}
\right|_{\alpha} ^{\alpha}}|} }
}}},
\end{equation}
or
\begin{equation}\label{eq:8x2}
x_{ij} = {\frac{
\sum\limits_{\beta
\in J_{r_{1},n} {\left\{ {i} \right\}}} {{\rm{cdet}} _{i} \left(
{\left(
{ {\bf A}^{ *} {\bf A} } \right)_{.i} \left(
{{ {\bf c}}_{.j}^{( 1)}} \right)}
\right)_{\beta} ^{\beta}
 }
 }{{{\sum\limits_{\beta
\in J_{r_{1},n}} {{\left| {{\bf A}^{ *} {\bf A}} \right| _{\beta} ^{\beta} }}
} }}}-{\frac{{{\sum\limits_{\beta
\in J_{r_{1},n} {\left\{ {i} \right\}}} {{\rm{cdet}} _{i} \left(
{\left(
{ {\bf A}^{ *} {\bf A} } \right)_{.i} \left(
{{ {\bf d}}_{.j}^{B}} \right)}
\right)_{\alpha} ^{\alpha} } }}}{{{{2\sum\limits_{\beta \in J_{r_{1},n}}{{\left|
{ {\bf A}^{ *} {\bf A} } \right| _{\beta} ^{\beta} }}\sum\limits_{\alpha \in I_{r_2,n}} {{\left|  {{\bf B}  {\bf B}^{
*}}
\right|_{\alpha} ^{\alpha}}|} }  }}}},
\end{equation}
where $ {\bf c}_{.j}^{( 1)}$ is the column vector of $  {\bf C}_{1}:={\bf A}^{*}{\bf C}$, and
\begin{align*}{{ {\bf d}}_{i\,.}^{ A}}=&\left[
\sum\limits_{\beta \in J_{r_{1},n} {\left\{ {i} \right\}}}
{{\rm{cdet}} _{i} \left( {\left( {{\bf A}^{
*}  {\bf A}} \right)_{.i} \left( {\bf{ c}}_{.k}^{(2)}\right)}
\right)_{\beta} ^{\beta}} \right]\in{\mathbb{H}}^{1 \times
n},\,\,\,\,k=1,\ldots,n,
   \\ {{{\bf d}}_{.\,j}^{B}}=&\left[
\sum\limits_{\alpha \in I_{r_{2},n} {\left\{ {j} \right\}}}
{{\rm{rdet}} _{j} \left( {\left(
{ {\bf B} {\bf B}^{ *} } \right)_{j.} \left( {\bf{ c}}_{l.}^{(2)}\right)}
\right)_{\alpha} ^{\alpha}}
\right]\in {\mathbb{H}}^{n \times
1},\,\,\,\,l=1,\ldots,n,
\end{align*}
 are the  row  vector and the column vector, respectively. ${\bf c}^{(2)}_{.k}$ and $ {\bf
 c}^{(2)}_{l.}$
 are the $k$th
column  and  the $l$th row  of $ { {\bf C}}_{2}:={\bf A}^{*}{\bf C}{\bf B}{\bf B}^{*}$.

\end{thm}
\begin{proof} Using Corollary \ref{cor:sol_AX} to the first term of (\ref{eq:sol_lyap}) and  Theorem \ref{theor:AXB=D} to the second term, we  get  (\ref{eq:8x1})-(\ref{eq:8x2}).

\end{proof}
2. Finally, consider the following matrix equation,
\begin{equation}\label{eq:lyap1}
 {\bf A}{\bf X}+ {\bf X}^{\ast}{\bf A}^{\ast}={\bf B},
\end{equation}
where ${\bf A}\in  {\mathbb{H}}^{m\times n}$ and  ${\bf B}\in  {\mathbb{H}}^{m\times m}$. Hodges \cite{hod}  found the explicit solution to (\ref{eq:lyap1}) and expressed it in terms of the Moore-Penrose inverse over a finite field. Djordjevi\'{c} \cite{djo} extended these results to the infinite dimensional settings.
Due to \cite{hod,djo}, the following lemma can be expanded  to  ${\mathbb{H}}$.
\begin{lem}\label{lem:lyap1}
The following statements are equivalent.
\begin{itemize}
  \item[(i)]  There exists a solution ${\bf X}\in  {\mathbb{H}}^{n\times m}$ to Eq. (\ref{eq:lyap1}).
  \item [(ii)] ${\bf B}^*={\bf B}$ or ${\bf R}_{A}{\bf B}{\bf R}_{A}={\bf 0}$.
\end{itemize}
In  that case,  the general solution  to (\ref{eq:lyap1}) can  be expressed  as the following,
\begin{equation}\label{eq:lyap1_sol}
 {\bf X}={\bf A}^{\dag}{\bf B}\left({\bf I}-\frac{1}{2}{\bf Q}_A\right)+{\bf L}_{A}{\bf Y}+{\bf P}_{A}{\bf Z}{\bf A}^{\ast},
\end{equation}
where ${\bf Z}\in  {\mathbb{H}}^{m\times m}$ satisfies ${\bf A}({\bf Z}+{\bf Z}^*){\bf A}^*={\bf 0}$ and  ${\bf Y}\in  {\mathbb{H}}^{m\times m}$ is arbitrary.
\end{lem}
\begin{proof}The proof is similar to (\cite{djo}, Theorem 2.2).

Note that, in \cite{djo}, the equation  ${\bf A}^{\ast}{\bf X}+ {\bf X}^{\ast}{\bf A}={\bf B}$ has been considered  instead (\ref{eq:lyap1}) with ${\bf A}$ and ${\bf B}$ as operators of Hilbert spaces. The result  obtained in \cite{djo} would be equal to (\ref{eq:lyap1_sol})  by substituting ${\bf A}$  with ${\bf A}^*$ and taking into account $${\bf P}_{A^*}=\left({\bf A}^*\right)^{\dag}{\bf A}^*=\left({\bf A}{\bf A}^{\dag}\right)^*={\bf Q}_{A},$$
so ${\bf Q}_{A^*}={\bf P}_{A}$, ${\bf L}_{A^*}={\bf I}-{\bf P}_{A^*}={\bf I}-{\bf Q}_{A}={\bf R}_{A}$, and ${\bf R}_{A^*}={\bf L}_{A}$.
\end{proof}
By putting ${\bf Z}={\bf Y}={\bf 0}$, we have the following partial solution to (\ref{eq:lyap1}),
\begin{equation}\label{eq:lyap1_psol}
 {\bf X}={\bf A}^{\dag}{\bf B}\left({\bf I}-\frac{1}{2}{\bf Q}_A\right).
\end{equation}

The following theorem on determinantal representations of (\ref{eq:lyap1_psol}) can be proven similar to Theorem \ref{th:sol_18}.

\begin{thm}\label{th:sol_19}Let ${\bf A}=\left(a_{ij}\right)\in  {\mathbb{H}}^{m\times n}_{r_{1}}$. Then   (\ref{eq:lyap1_psol})  has the  determinantal representations
\begin{equation}\label{eq:9x1}
x_{ij} = {\frac{
\sum\limits_{\beta
\in J_{r_{1},n} {\left\{ {i} \right\}}} {{\rm{cdet}} _{i} \left(
{\left(
{ {\bf A}^{ *} {\bf A} } \right)_{.i} \left(
{{ {\bf b}}\,_{.j}^{( 1)}} \right)}
\right)_{\beta} ^{\beta}
 }
 }{{{\sum\limits_{\beta
\in J_{r_{1},n}} {{\left| {{\bf A}^{ *} {\bf A}} \right| _{\beta} ^{\beta} }}
} }}}-
\frac{{{\sum\limits_{\alpha \in I_{r_1,m} {\left\{ {j}
\right\}}} {{{\rm{rdet}} _{j} {\left( {( {\bf A} {\bf A}^{ *}
)_{j.}\, ({{{\bf d}}_{i.}} )}
\right)  _{\alpha} ^{\alpha} } }}} }}{{{{2\sum\limits_{\beta \in J_{r_{1},n}}{{\left|
{ {\bf A}^{ *} {\bf A} } \right| _{\beta} ^{\beta} }}\sum\limits_{\alpha \in I_{r_1,m}} {{\left|  {{\bf A}  {\bf A}^{
*}}
\right|_{\alpha} ^{\alpha}}} }
}}},
\end{equation}
or
\begin{equation}\label{eq:9x2}
x_{ij} = {\frac{
\sum\limits_{\beta
\in J_{r_{1},n} {\left\{ {i} \right\}}} {{\rm{cdet}} _{i} \left(
{\left(
{ {\bf A}^{ *} {\bf A} } \right)_{.i} \left(
{{ {\bf b}}\,_{.j}^{( 1)}} \right)}
\right)_{\beta} ^{\beta}
 }
 }{{{\sum\limits_{\beta
\in J_{r_{1},n}} {{\left| {{\bf A}^{ *} {\bf A}} \right| _{\beta} ^{\beta} }}
} }}}-{\frac{{{\sum\limits_{\beta
\in J_{r_{1},n} {\left\{ {i} \right\}}} {{\rm{cdet}} _{i} \left(
{\left(
{ {\bf A}^{ *} {\bf A} } \right)_{.i} \left(
{{ {\bf d}}_{.j}} \right)}
\right)_{\alpha} ^{\alpha} } }}}{{{{2\sum\limits_{\beta \in J_{r_{1},n}}{{\left|
{ {\bf A}^{ *} {\bf A} } \right| _{\beta} ^{\beta} }}\sum\limits_{\alpha \in I_{r_1,m}} {{\left|  {{\bf A}  {\bf A}^{
*}}
\right|_{\alpha} ^{\alpha}}} }  }}}},
\end{equation}
where ${ {\bf b}}\,_{.j}^{( 1)}$ is the column vector of $ { {\bf B}}_{1}:={\bf A}^{*}{\bf B}$, and
\begin{gather*} {\bf d}_{i\,.}=\left[
\sum\limits_{\beta \in J_{r_{1},n} {\left\{ {i} \right\}}}
{{\rm{cdet}} _{i} \left( {\left( {{\bf A}^{
*}  {\bf A}} \right)_{.i} \left( {\bf{ b}}_{.k}^{(2)}\right)}
\right)_{\beta} ^{\beta}} \right]\in{\mathbb{H}}^{1 \times
n},\,\,\,\,k=1,\ldots,n,
   \\ {\bf d}\,_{.j}=\left[
\sum\limits_{\alpha \in I_{r_{1},m} {\left\{ {j} \right\}}}
{{\rm{rdet}} _{j} \left( {\left(
{ {\bf A} {\bf A}^{ *} } \right)_{j.} \left( {\bf{ b}}_{l.}^{(2)}\right)}
\right)_{\alpha} ^{\alpha}}
\right]\in {\mathbb{H}}^{n \times
1},\,\,\,\,l=1,\ldots,n,
\end{gather*}
 are the  row  vector and the column vector, respectively. ${{{\bf b}}^{(2)}_{.k}}$ and ${ {\bf
 b}}^{(2)}_{l.}$
 are the $k$th
column  and  the $l$th row  of $ { {\bf B}}_{2}:={\bf A}^{*}{\bf B}{\bf A}{\bf A}^{*}$.

\end{thm}
\section{Examples}
In this section, we give an example to illustrate our results.

1. Consider the  matrix equation
\begin{equation}\label{ex:tss1}
 {\bf A}_{1}{\bf X}_{1}{\bf B}_{1}+ {\bf A}_{2}{\bf X}_{2}{\bf B}_{2}={\bf C},
\end{equation}
with given matrices
\begin{align*}
{\bf A}_1=\begin{bmatrix}
                                     i&1 \\-1&i \\k& -j
 \end{bmatrix},~{\bf A}_2=\begin{bmatrix}i\\
                                          j \\
                                          k   \end{bmatrix},~{\bf B_1}=\begin{bmatrix}
             i\\
             k
 \end{bmatrix},~{\bf B_2}=\begin{bmatrix}j \\
                                      -i
                                    \end{bmatrix},~{\bf C}=\begin{bmatrix}1\\i\\2j
 \end{bmatrix}.
 \end{align*}
By this given matrices, the consistency conditions of  (\ref{eq:cond1}) from Lemma \ref{lem:cond1} are fulfilled. So, the system (\ref{ex:tss1}) is resolvable.
Using determinantal representations (\ref{eq:det_repr_A*A})-(\ref{eq:det_repr_AA*}) for computing Moore-Penrose inverses, we find that
 \begin{align*}
 &{\bf A}_1^{\dag}=\frac{1}{6}\begin{bmatrix}
        -i &-1& -k \\
         1&   -i&j
      \end{bmatrix},~{\bf R}_{A_{1}}=\frac{1}{3}\begin{bmatrix}
               2 &  i&-j \\
                  -i &  2&-k\\
                  j&k&2
               \end{bmatrix},~{\bf B}_1^{\dag}=\frac{1}{2}\begin{bmatrix}
                    -i &
                     -k
                  \end{bmatrix},\\&{\bf B_2}^{\dag}=\frac{1}{2}\begin{bmatrix}
                    -j &
                     i
                  \end{bmatrix},~{\bf M}=\frac{1}{3}\begin{bmatrix}
               2 & 2i& 4j
               \end{bmatrix},~{\bf M}^{\dag}=\frac{1}{4}\begin{bmatrix}
                         1& -i & -j
                        \end{bmatrix}.
 \end{align*} So, ${\bf L}_{B_{1}}=$ and ${\bf N}=0$. By putting free
matrices ${\bf U},~{\bf V},~{\bf Z}$, and ${\bf W}$ as zero-matrices,   we first obtain the pair solution by direct matrix multiplications
\begin{align}\label{eq:ex_sol_x1}
{\bf X}_{1}=&{\bf A}^{\dag}_{1}{\bf C} {\bf B}^{\dag}_{1}-
{\bf A}^{\dag}_{1}{\bf A}_{2}{\bf M}^{\dag}{\bf C} {\bf B}^{\dag}_{1}=\frac{1}{8}\begin{bmatrix}
               1 & j\\i& k
               \end{bmatrix},\\
\nonumber{\bf X}_{2}=&{\bf M}^{\dag}{\bf C}{\bf B}_{2}^{\dag}=\frac{3}{4}\begin{bmatrix}
               -j & i
               \end{bmatrix}.
\end{align}
Now, we find the solution to (\ref{ex:tss1}) by our new proposed approach, namely, by Cramer's Rule thanks to Theorem \ref{th:sol_1}.
Since ${\bf A}^{\ast}_{1}{\bf C} {\bf B}^{\ast}_{1}={\bf 0}$, then $x^{(11)}_{ij}=0$ for all $i,j=1,2$. Therefore, $\rank{\bf A}_1=\rank{\bf A}_2=\rank{\bf B}_1=\rank{\bf B}_2=1$, and
 \begin{align*}
 &{\bf M}^*{\bf M}=\left[\frac{8}{3}\right],~{\bf C}_2=\begin{bmatrix}
               -4i & -4k
               \end{bmatrix},~{\bf B}_1{\bf B}_1^*=\begin{bmatrix}
               1 & j\\
               -j&1
               \end{bmatrix},~{\tilde{\bf A}}_2=\begin{bmatrix}
               -i\\
               1
               \end{bmatrix}.\end{align*}
 So,
 \begin{equation}\label{ex:x_1}
 x_{11}^{(1)}=- x_{11}^{(12)}=-\frac{(-i)(-4i)}{6\cdot\frac{8}{3}\cdot2}=\frac{1}{8}.
 \end{equation}
Hence, $x^{(1)}_{11}$ obtained by Cramer's Rule (\ref{ex:x_1}) and by the matrix method (\ref{eq:ex_sol_x1}) are equal.
Similarly, we can obtain for   $x^{(1)}_{12},~x^{(1)}_{21},~x^{(1)}_{22},~x^{(2)}_{11}$, and $x^{(2)}_{12}$.

2. Let
us consider the  matrix equation
\begin{equation}\label{ex:eq}
 {\bf A}{\bf X}+ {\bf X}^{\ast}{\bf A}^{\ast}={\bf B},
\end{equation}
where
\[{\bf A}=\begin{bmatrix}
 2 & j \\
   -k & i\\
    i & k \\
\end{bmatrix},\,\,{\bf B}=\begin{bmatrix}
  2 & j & -k \\
  -j & 1 & i\\
  k & -i & 2
\end{bmatrix}.\] Since,  ${\bf B}^*={\bf B}$,  then, by Lemma \ref{lem:lyap1}, the equation (\ref{ex:eq}) is consistent.
Since
${\bf A}^{*}{\bf A}=\begin{bmatrix}
  6 & 4j  \\
  -4j & 3
\end{bmatrix}$ and $\det {\bf A}=2$,
 then $\rank{\bf A}=2$.
By Theorem \ref{theor:det_repr_MP} and Lemma \ref{eq:det_repr_proj_Q}, one can find,
$${\bf A}^{\dag}=\frac{1}{2}\begin{bmatrix}
  2 & -k & i\\
  2j & -2i & -2k
\end{bmatrix},\,\,{\bf Q}_A=\frac{1}{2}\begin{bmatrix}
  2 & 0 & 0\\
  0 & 1 & j\\
 0 & -j & 1
\end{bmatrix}.$$

First, we can find   the   solution to (\ref{ex:eq}) by direct calculation. By (\ref{eq:lyap1_psol}),
\begin{align}\nonumber{\bf X}=&{\bf A}^{\dag}{\bf B}-\frac{1}{2}{\bf A}^{\dag}{\bf B}{\bf Q}_A=0.5\begin{bmatrix}
  4-i-j & 1+2j-k & 2i-j-2k\\
 2+4j+2k & -2-2i +2j& 2-2i-4k
\end{bmatrix}-
\\\nonumber-& 0.25\begin{bmatrix}
 4-i-j & -i+j-1.5k &1+ 0.5i-j-k\\
 2+4j+2k & -1-3i +k& 2-i+j-k
\end{bmatrix}=
\\=&0.25\begin{bmatrix}
  4-i-j &2 +i+3j-0.5k &-1- 2.5i+j-3k\\
2+4j+2k & -3-i +4j-k& 2-3i-j-7k
\end{bmatrix}.\label{eq:sys_sol1_x1_ex}
\end{align}

Now,  we  find   the   solution to (\ref{ex:eq}) by
it's determinantal representation  (\ref{eq:9x1}).
Since,
\begin{multline*}
{\bf A}{\bf A}^{*}=\begin{bmatrix}
  5 &3k&-3i  \\
 -3k & 2&2j\\
 3i & -2j&2
\end{bmatrix},\\
 {\bf B}_{1}:={\bf A}^{*}{\bf B}=\begin{bmatrix}
  4+i+j &-1+ 2j+k&-2i+j-2k  \\
 1-2j+k & 1-i+j&1+i-2k
\end{bmatrix},\\{\bf B}_{2}={\bf A}^{*}{\bf B}{\bf A}{\bf A}^{*}=\begin{bmatrix}
  29-i-j &-i+ j+18k&-1-18i-k  \\
 2-19j+2k & -1-12i+k&-i-j-12k
\end{bmatrix},\end{multline*}
and
\begin{multline*}
d^{(1)}_{11}=
\sum\limits_{\beta \in J_{2,2} {\left\{ {1} \right\}}}
{{\rm{cdet}} _{1} \left( {\left( {{\bf A}^{
*}  {\bf A}} \right)_{.1} \left( {\bf{ b}}_{.1}^{(2)}\right)}
\right)_{\beta} ^{\beta}}=\\{\rm{cdet}} _{1}\begin{bmatrix}
  29-i-j & 4j  \\
 2-19j+2k & 3
\end{bmatrix}=11-11i-11j,
\end{multline*}
\begin{multline*}
d^{(1)}_{12}=
\sum\limits_{\beta \in J_{2,2} {\left\{ {1} \right\}}}
{{\rm{cdet}} _{1} \left( {\left( {{\bf A}^{
*}  {\bf A}} \right)_{.1} \left( {\bf{ b}}_{.2}^{(2)}\right)}
\right)_{\beta} ^{\beta}}=\\{\rm{cdet}} _{1}\begin{bmatrix}
 -i-j+18k & 4j  \\
 -1-12i+k & 3
\end{bmatrix}=-7i+7j+6k,\end{multline*}\begin{multline*}
d^{(1)}_{13}=
\sum\limits_{\beta \in J_{2,2} {\left\{ {1} \right\}}}
{{\rm{cdet}} _{1} \left( {\left( {{\bf A}^{
*}  {\bf A}} \right)_{.1} \left( {\bf{ b}}_{.3}^{(2)}\right)}
\right)_{\beta} ^{\beta}}=\\{\rm{cdet}} _{1}\begin{bmatrix}
 -1-18i-k & 4j  \\
 -i-j-12k & 3
\end{bmatrix}=-7-6i-7k,
\end{multline*}
and
$$({\bf A}{\bf A}^{*})_{1.}\left({\bf d}_{1.}^{(1)}\right)=\begin{bmatrix}
  11-11i-11j &-7i+7j+6k&-7-6i-7k  \\
 -3k & 2&2j\\
 3i & -2j&2
\end{bmatrix},$$
 then
\begin{align*}
&x_{11} =\\&= {\frac{
\sum\limits_{\beta
\in J_{2,2} {\left\{ {1} \right\}}} {{\rm{cdet}} _{1} \left(
{(
{ {\bf A}^{ *} {\bf A} } )_{.1} (
{{ {\bf b}}_{.1}^{( 1)}} )}
\right)_{\beta} ^{\beta}
 }
 }{{{\sum\limits_{\beta
\in J_{2,2}} {{\left| {{\bf A}^{ *} {\bf A}} \right| _{\beta} ^{\beta} }}
} }}}-
\frac{{{\sum\limits_{\alpha \in I_{2,3} {\left\{ {1}
\right\}}} {{{\rm{rdet}} _{1} {\left( {( {\bf A} {\bf A}^{ *}
)_{1.}\, ({{{\bf d}}_{1.}^{( 1)}} )}
\right)  _{\alpha} ^{\alpha} } }}} }}{{{{\sum\limits_{\beta \in J_{2,2}}{{\left|
{ {\bf A}^{ *} {\bf A} } \right| _{\beta} ^{\beta} }}\sum\limits_{\alpha \in I_{2,3}} {{\left|  {{\bf A}  {\bf A}^{
*}}
\right|_{\alpha} ^{\alpha}}} }
}}}=\\&=\frac{1}{4}{\rm{cdet}} _{1}\begin{bmatrix}
   4+i+j & 4j  \\
  1-2j+k & 3
\end{bmatrix}-\frac{1}{4}\left({\rm{rdet}} _{1}\begin{bmatrix}
  11-11i-11j &-7i+7j+6k  \\
 -3k & 2
\end{bmatrix}+\right.\\&+\left.{\rm{rdet}} _{1}\begin{bmatrix}
  11-11i-11j &-7-6i-7k  \\
 3i & 2
\end{bmatrix}\right)=\frac{1}{4}\left(4-i-j\right).
\end{align*}
So, $x_{11}$ obtained by Cramer's rule and the matrix method (\ref{eq:sys_sol1_x1_ex}) are equal.

Similarly, we can obtain for all  $x_{ij}$, $i=1,2$ and $j=1,2,3$.

Note that we used Maple with the package CLIFFORD in the calculations.

\section{Conclusions}
Within the framework of the theory of row-column determinants, we have derived explicit formulas for determinantal representations  (analogs of Cramer's Rule) of  solutions
 to the quaternion  two-sided generalized Sylvester matrix equation $  {\bf A}_{1}{\bf X}_{1}{\bf B}_{1}+ {\bf A}_{2}{\bf X}_{2}{\bf B}_{2}={\bf C}$ and its all special cases when its first term or both terms are one-sided. Finally,  determinantal representations of two like-Lyapunov equations have been obtained. To accomplish that goal,  determinantal representations of the Moore-Penrose  inverse  previously introduced by the author have been used.


\begin{thebibliography}{40}
\bibitem{baks} J.K.~Baksalary, R.~Kala,  The matrix equation $AX - YB = C$, Linear algebra Appl. 25 (1979) 41–43.
\bibitem{bak}  J.K.~Baksalary,  R.~Kala,  The  matrix  equation  $AXB  -  CYD  =  E$,  Linear  Algebra  Appl.  30
(1980)  141-147.

  \bibitem{chen}C. Chen, Dan Schonfeld, Pose estimation from multiple cameras based on Sylvester's equation,
Comput. Vis. Image Underst. 114 (2010) 652-666.






\bibitem{csong1} C.~Song, G.~Chen, On solutions of matrix equation $XF - AX = C$ and $XF - A\tilde{X} = C$ over
quaternion field, J. Appl. Math. Comput. 37 (2011) 57-68.

\bibitem{csong2} C.Q.~Song,  G.L.~Chen,  Q.B.~Liu,  Explicit  solutions  to  the  quaternion  matrix  equations  $X -
AXF = C$ and $X - A\bar{X}F = C$, Int. J. Comput. Math. 89 (2012) 890-900.



 \bibitem{deh} M.~Dehghan,  M.~Hajarian,  The generalized Sylvester matrix equations over
the generalized bisymmetric and skew-symmetric matrices. Int. J. Syst. Sci.
43 (2012) 1580-1590.

\bibitem{djo}D.S.~Djordjevi\'{c}, Explicit solution of the operator equation $A^*X + X^*A=B$. J. Comput. Appl. Math. 200 (2007) 701-704.
 \bibitem{he2} Z.H.~He, Q.W.~Wang, Y.~Zhang, A system of quaternary coupled Sylvester-type real quaternion matrix equations. Automatica 87 (2018) 25-31.
  \bibitem{he3}Z.H.~He, Q.W.~Wang, A System of periodic discrete-time coupled Sylvester quaternion matrix equations. Algebra Colloq. 24(1) (2017) 169-180.

\bibitem{hod}
 J.H.~Hodeges, Some matrix equations over a finite field, Ann. Mat. Pura Appl. 44 (1) (1957) 245-250.

\bibitem{fut}V.~Futorny, T.~Klymchuk, V.V.~Sergeichuk, Roth's solvability criteria for the matrix equations $  A X-\widehat{ X} B= C$ and $ X-  A\widehat{ X} B= C$ over the skew field of quaternions with an involutive automorphism $q\rightarrow \widehat{q}$, Linear algebra Appl. 510 (2016) 246-258.
    
    \bibitem{kyr_nov}I.~Kyrchei, Cramer's rule for generalized inverse solutions. In: I. Kyrchei (Ed.), Advances in Linear Algebra Research, pp. 79--132,  Nova Sci. Publ., New York,  2015.


 \bibitem{kyr2} I.~Kyrchei, Cramer's rule for quaternionic systems of linear equations, J. Math. Sci. 155(6) (2008) 839--858.


 \bibitem{kyr3}    I.~Kyrchei, The theory of the column and row determinants in a quaternion linear algebra.  In: Albert R. Baswell (Ed.), Advances in Mathematics Research 15, pp. 301-359,  Nova Sci. Publ., New York, 2012.

\bibitem{kyr4}    I.~Kyrchei,
Determinantal representations of the Moore-Penrose inverse over the quaternion skew field and corresponding Cramer's rules, J. Math. Sci. 180(1) (2012) 23-33.

 \bibitem{kyr5}   I.~Kyrchei,
Explicit determinantal representation formulas of W-weighted Drazin inverse solutions of some matrix equations over the quaternion skew field, Math. Probl. Eng. Art. ID 8673809 (2016) 13 pages.

 \bibitem{kyr6}I.~Kyrchei,
Determinantal representations of the Drazin inverse over the quaternion skew field with applications to some matrix equations, Appl. Math. Comput. 238 (2014) 193-207.

 \bibitem{kyr7}I.~Kyrchei,
Determinantal representations of the W-weighted Drazin inverse over the quaternion skew field, Appl. Math. Comput. 264 (2015) 453-465.

\bibitem{kyr8}I.~Kyrchei, Explicit determinantal representation formulas for the solution of the two-sided restricted quaternionic matrix equation, J. Appl. Math. Comput. 58(1-2)  (2018) 335--365.



\bibitem{kyr9}I.~Kyrchei,
 Determinantal representations of the Drazin and W-weighted Drazin inverses over the quaternion skew field with applications.  In: Sandra Griffin (Ed.), Quaternions: Theory and Applications, pp.201-275, Nova Sci. Publ., New York,  2017.




\bibitem{kyr10}I.~Kyrchei,
 Weighted singular value decomposition and determinantal representations of the quaternion weighted Moore-Penrose inverse, Appl. Math. Comput.  309 (2017)  1-16.


 \bibitem{kyr11}I.~Kyrchei,
 Determinantal representations of solutions to systems of quaternion matrix equations, Adv. Appl. Clifford Algebras 28 (2018) 23 pages.

 \bibitem{kyr12}     I.I.~Kyrchei, Cramer's rules for the system of two-sided matrix equations and of its special cases,  In: H. A. Yasser (Ed.), Matrix Theory, pp. 3-20, IntechOpen,  2018.

\bibitem{kyr13}  I.I.~Kyrchei,   Determinantal representations of the quaternion weighted Moore-Penrose inverse and its applications, In: Albert R.B. (Ed.), Advances in Mathematics Research 23, pp.35--96, Nova Sci. Publ., New York, 2017.




 \bibitem{liao}  A.~Liao,  Z.~Bai, Y.~Lei,  Best approximate solution of matrix equation $AXB+
CYD = E$, SIAM J. Matrix Anal. Appl. 27 (2006) 675-688.
\bibitem{lip} H.~Liping, The matrix equation $AXB - GXD = E$ over the quaternion field, Linear Algebra Appl.
234 (1996) 197-208.

\bibitem{mac}A.A.~Maciejewski, C.A.~Klein, Obstacle avoidance for kinematically redundant manipulators in dynamically varying environments, The International Journal of Robotics Research 4(3) (1985) 109-117.
\bibitem{mans}A. Mansour, Solvability of $AXB - CXD = E$ in the operators algebra $B(H)$, Lobachevskii J. Math. 31(3) (2010) 257-261.


 \bibitem{peng} Z.Y.~Peng, Y.X.~Peng,  An eficient iterative method for solving the matrix
equation $AXB +CYD = E$, Numer. Linear Algebra Appl. 13 (2006) 473-485.
\bibitem{pia} F.~Piao, Q.~Zhang, Z.~Wang, The solution to matrix equation $AX+X^TC= B$.
J. Franklin. Inst. 344 (2007) 1056--1062.


 \bibitem{reh1}A.~Rehman, Z.H.~He, Y.~Zhang, Constraint generalized Sylvester matrix equations. Automatica 69 (2016) 60-64.
   \bibitem{reh2}
  A.~Rehman, Q.W.~Wang, Z.H.~He, Solution to a system of real quaternion matrix equations
encompassing $\eta$-Hermicity. Appl. Math. Comput. 265 (2015) 945-957.

 \bibitem{reh3}
  A.~Rehman, Q.W.~Wang, I.~Ali, M.~Akram, M.O.~Ahmad,
A constraint system of generalized Sylvester quaternion matrix equations. Adv. Appl Clifford Algebras 27 (4) (2017) 3183-3196.





\bibitem{shah}A. Shahzad, B.L. Jones, E.C. Kerrigan, G.A. Constantinides, An efficient algorithm for the solution
of a coupled Sylvester equation appearing in descriptor systems, Automatica 47 (2011) 244-248.

 \bibitem{shi}S.Y.~Shi,  Y.~Chen,  Least squares solution of matrix equation $AXB +CYD =
E$, SIAM J. Matrix Anal. Appl. 24 (2003) 802-808.
\bibitem{sim}\c{S}im\c{s}ek,~S., Sarduvan,~M., \"{O}zdemir,~H.:  Centrohermitian and skew-centrohermitian solutions to the minimum residual and matrix nearness problems of the quaternion matrix equation   $(AXB,DXE)=(C,F)$.  Adv. Appl. Clifford Algebras 27(3)  (2017) 2201--2214.
\bibitem{song1}
 G.J.~Song, Q.W.~Wang, H.X.~Chang,  Cramer rule for the unique solution of
restricted matrix equations over the quaternion skew field, Comput. Math. Appl.
61 (2011) 1576-1589.

\bibitem{song2}G.J.~Song,  C.Z.~Dong, New results on condensed Cramer's rule for the general solution to some restricted quaternion matrix equations, J. Appl. Math. Comput. 53 (2017) 321-341.

\bibitem{song3} G.J.~Song, Q.W.~Wang, Condensed Cramer rule for some restricted quaternion
linear equations, Appl. Math. Comp. 218 (2011) 3110-3121.

\bibitem{song5}
 G.J.~Song,  Q.W.~Wang,   S.W. Yu, Cramer's rule for a system of quaternion matrix equations with applications, Appl. Math. Comp. 336 (2018) 490--499.


 \bibitem{syrm1} V.L. Syrmos, F.L. Lewis, Output feedback eigenstructure assignment using two Sylvester equations,
IEEE Trans. Automat. Control 38 (1993) 495-499.
\bibitem{syrm2} V.L. Syrmos, F.L. Lewis, Coupled and constrained Sylvester equations in system design, Circuits
Systems Signal Process. 13 (6) (1994) 663-694.


\bibitem{varg}A. Varga, Robust pole assignment via Sylvester equation based state feedback parametrization, in:
IEEE International Symposium on Computer-Aided Control System Design, CACSD 2000, vol. 57,
2000, pp. 13-18.
  \bibitem{wang04}     Q.W.~Wang, A system of matrix equations and a linear matrix equation over arbitrary regular rings
with identity, Linear  Algebra  Appl. 384 (2004) 43-54.

\bibitem{wang09} Q.W.~Wang, J.W.~van~der~Woude, H.X.~Chang,
A system of real quaternion matrix equations with
applications, Linear Algebra Appl. 431(1) (2009) 2291-2303.

\bibitem{wua}  A.G. Wu, F.Zhu,  G.R.Duan,  Y.Zhang,  Solving the generalized Sylvester
matrix equation $AV + BW = EVF$ via a Kronecker map, Appl. Math. Lett.
21(2008) 1069-1073.







\bibitem{xu}G.~Xu,  M.~Wei, D.~Zheng, On  solutions  of  matrix  equation $AXB+
CYD = E$,  Linear Algebra Appl.  279 (1998) 93-109.



\bibitem{yuan1} S.F.~Yuan,  Q.W.~Wang, Y.B.~Yu, Y.~Tian, On hermitian solutions of the split
quaternion matrix equation $AXB +CXD= E$, Adv. Appl. Clifford Algebras 27(4) (2017) 3235-3252.

 \bibitem{yuan2}S.F.~Yuan, Q.W.~Wang,  Two special kinds of least squares solutions for the quaternion matrix equation
$AXB+CXD=E$, Electron. J. Linear Algebra 23 (2012) 57-74.
\bibitem{yuan3} S.F.~Yuan, A.P.~Liao, Least squares solution of the quaternion matrix equation $X - A\bar{X}B = C$ with
the least norm, Linear Multilinear Algebra 59 (2011) 985-998.

 \bibitem{zhan}
  X.~Zhang, A system of generalized Sylvester quaternion matrix equations and its applications.
Appl. Math. Comput. 273 (2016) 74-81.

\bibitem{zhang1} Y.N. Zhang, D.C. Jiang, J. Wang, A recurrent neural network for solving Sylvester equation with
time-varying coeficients, IEEE Trans. Neural Netw. 13 (5) (2002) 1053-1063.





\bibitem{zhou}B. Zhou, L. James, G.R. Duan, On Smith-type iterative algorithms for the Stein matrix equation,
Appl. Math. Lett. 22 (7) (2009) 1038-1044.

























\end{thebibliography}
\end{document}